\documentclass[12pt, reqno]{amsart}

\advance \topmargin by -\headheight
\advance \topmargin by -\headsep
\evensidemargin \oddsidemargin
\usepackage{graphicx} 
\usepackage{amsfonts}
\usepackage{amsmath}
\usepackage{amssymb}
\usepackage{amsthm}
\usepackage{tikz-cd}

\usepackage{amsmath, amsthm, amssymb, amsfonts}
 \usepackage{hyperref}
\usepackage{enumerate}
\usepackage{url}
\usepackage{dsfont}
\usepackage{mathrsfs}

\usepackage{bm}

\usepackage{xcolor}

\usepackage{fancyvrb}

\usepackage[english]{babel}
\usepackage{color}
\usepackage{enumerate}

\newcommand{\G}{\mathbb{G}}
\renewcommand{\P}{\mathbb{P}}

\newcommand{\id}{\textrm{id}}
\newcommand{\C}{\mathbb{C}}

\newcommand{\Vsum}{\mathcal{V}_{\text{sum}}}
\renewcommand{\vec}[1]{\mathbf{#1}}

\newcommand{\Z}{\mathbb{Z}}
\newcommand{\rk}{\mathrm{rk}}

\renewcommand{\vec}{\bm}

\title[Sum-product phenomenon for algebraic groups and Bremner's conjecture]{Uniform sum-product phenomenon for algebraic groups and Bremner's conjecture}
\author{Joseph Harrison, Akshat Mudgal, Harry Schmidt}
\address{Mathematics Institute, Zeeman Building, University of Warwick, Coventry CV4 7AL, United Kingdom}
\email{joseph.s.harrison@warwick.ac.uk}
\email{Akshat.Mudgal@warwick.ac.uk}
\email{Harry.Schmidt@warwick.ac.uk}


\newtheorem{theorem}{Theorem}[section]
\newtheorem{lemma}[theorem]{Lemma}
\newtheorem{corollary}[theorem]{Corollary}
\newtheorem{proposition}[theorem]{Proposition}
\newtheorem{conjecture}[theorem]{Conjecture}

\newtheorem{thm}[theorem]{Theorem}
\newtheorem{lem}[theorem]{Lemma}

\newtheorem{cor}[theorem]{Corollary}

\theoremstyle{remark}
\newtheorem{remark}[theorem]{Remark}
\newtheorem{example}[theorem]{Example}

\theoremstyle{definition}
\newtheorem{definition}[theorem]{Definition}

\numberwithin{equation}{section}

\begin{document}

 \maketitle

\begin{abstract} 
In this paper we combine methods from additive combinatorics and Diophantine geometry to study the generalised sum-product phenomenon in algebraic groups. As an application of this circle of ideas, we resolve  a conjecture of Bremner on arithmetic progressions in coordinates of elliptic curves, along with various other generalisations studied in the literature. 

We also prove a uniform Bourgain--Chang-type sum-product estimate for general $1$-dimensional algebraic groups $G$ over $\mathbb{C}$. Using these ideas, we provide an alternative solution to a problem of Bays--Breuillard. Furthermore, we show an Elekes--Szab\'{o} type result in the same setting for sets with small doubling, improving upon an earlier result of Bays--Breuillard when $G$ is not $\mathbb{G}_a$. Our power saving here can be shown to be quantitatively optimal.

We use a combination of deep, classical results in Diophantine geometry due to David--Philippon, Laurent and Evertse--Schmidt--Schlickewei along with the recent breakthrough work on the weak Polynomial Freiman--Ruzsa conjecture over integers due to Gowers--Green--Manners--Tao.

\end{abstract}

\section{Introduction} \label{introduction}

 Many questions in number theory concern an incongruence between two distinct arithmetic structures.  Bremner \cite{Bremner} made this observation in the course of his investigations into the length of arithmetic progressions in the coordinates of the rational points of an elliptic curve.  Here, additive structure, represented by an arithmetic progression, and the group structure on the elliptic curve should not correlate with each other, which led him to suspect that the length of a possible arithmetic progression should be bounded solely in terms of the rank of the curve. He confirmed his suspicions partly in work  with Silverman and Tzanakis \cite{BremnerSilverman}. 
 
 Another example of this phenomenon is the infamous sum--product conjecture in combinatorial number theory due to Erd\H{o}s--Szemer\'{e}di \cite{Erdos--Szemeredi_on_sums_of_and_products_of_integers}. This concerns expansion of finite sets of integers under the operation of taking sums or products, a manifestation of the incompatibility of additive and multiplicative structure. Progress towards this type of problem has led to a variety of applications in number theory and harmonic analysis. 
 
 A third example of this phenomenon concerns expansion of arbitrary sets of real numbers under sufficiently non-degenerate polynomial maps. This was first investigated by Elekes--R\'{o}nyai \cite{ER2000} and Elekes--Szab\'{o} \cite{ES2012}, and since their work, this subject has seen significant activity, in part due to its connections to questions in combinatorial geometry and topics in model theory. In this paper, we unify these a priori disparate themes and prove effective, quantitative results regarding them.

 We first turn towards the question of Bremner which concerns upper bounds on possible lengths of arithmetic progressions in the coordinates of the rational points of an elliptic curve. Replacing the additive structure by the multiplicative one, one may similarly suspect that the length of a geometric progression should be bounded, see work of Bremner--Ulas \cite{Brmenergeom}. Another line of inquiry, in the same spirit, concerns the length of the longest sequence of consecutive squares in the coordinates of rational points on elliptic curves, see work of Kamel--Sadek  \cite{Kamelsquares}.
 
 As a straightforward consequence of the methods discussed in this article, we confirm Bremner's speculation \cite[\S5]{Bremner} in the following general result.

 \begin{thm} \label{thmBremner} There is an effectively computable constant $C \geq 1$ with the following property. Let $E$ be an elliptic curve in Weierstrass form 
 \begin{align}\label{Weierstrass}
      y^2 = x^3 + ax + b,~~  a,b \in \mathbb{Q},
 \end{align}
 and let $r$ be the rank of $E(\mathbb{Q})$. Let $X = \{x(P): P \in E(\mathbb{Q})\}$ and $Y = \{y(P): P \in E(\mathbb{Q})\}$. Let $A$ be either an arithmetic progression, a geometric progression or a set of the form 
 \[ \{ u^2, (u+d)^2, (u+2d)^2, \dots, (u+dl)^2 \}, \]
 with $u,d\in \mathbb{Q}$, and $l \in \mathbb{N}$. If $A \subseteq X$ or $A \subseteq Y$, then $|A| \leq C^{1 +r}$.
\end{thm}

We note that the constant $C$ in Theorem \ref{thmBremner} does not depend on $a,b$, and for families of elliptic curves of bounded rank we obtain a uniform bound. It is to this day unclear whether there exist elliptic curves of arbitrary large rank. Recently, an elliptic curve of rank at least 29 was discovered \cite{rank_29_elliptic_curve, rank_records}. It is straightforward to see that Siegel's theorem for $S$-integral points \cite[Theorem D.9.1]{HindrySilverman} on elliptic curves implies that there can be no infinite arithmetic or geometric sequence or sequence of squares in $E(\mathbb{Q})$. However, it does not provide a uniform bound, even for a fixed elliptic curve, since the number of primes that need inverting depends on the particular sequence. 

 For arithmetic progressions, Garcia--Fritz and Pasten \cite{GFP2021} prove a bound of the form $C^{1 + r}$ with $C$ depending on $E$ but it is conceivable that their methods could lead to a uniform bound, had they used a uniform version of Mordell--Lang, e.g., \cite{DavidPhilippon}. However, for geometric progressions and consecutive squares no such general bounds seem to be known, even when allowing a dependence on $E$. We prove a more general version of Theorem \ref{thmBremner} (Corollary \ref{corBremner}) for correspondences and finite rank groups that applies to a plethora of similar sequences. Theorem \ref{thmBremner} can also be formulated in terms of rational points on a surface of general type in a high dimensional projective space. As a consequence we determine the Zariski closure of the rational points of certain projective surfaces (see Theorem \ref{BombieriLang}).

We now turn towards the sum-product phenomenon. Thus, 
given finite sets $A,B \subseteq \mathbb{C}$, we define the sumset and the product set as
\[ A+B = \{ a+ b : a \in A, b \in B\} \ \ \text{and} \ \ A\cdot B = \{ab : a \in A, b \in B\}.   \]
It is expected that if $|A+A|$ is small in terms of $|A|$, then $A$ should be additively structured, and if $|A \cdot A|$ is small in terms of $|A|$, then $A$ should be multiplicatively structured. Speculating that these two types of structures should not coexist simultaneously, Erd\H{o}s--Szemer\'{e}di \cite{Erdos--Szemeredi_on_sums_of_and_products_of_integers} conjectured that for any finite set $A \subseteq \mathbb{C}$, either $A+A$ or $A \cdot A$ must have size close to $|A|^2$. More generally, writing 
\[ gA = \{a_1 + \dots + a_g : a_1, \dots, a_g \in A\} \ \ \text{and} \ \ A^{(g)} = \{ a_1 \dots a_g : a_1, \dots, a_g \in A\}, \]
for any $g \in \mathbb{N}$, Erd\H{o}s--Szemer\'{e}di conjectured the following. 

\begin{conjecture} \label{erdszem}
    For any $g \in \mathbb{N}$, any $\varepsilon>0$ and any set $A \subseteq \mathbb{C}$, one should have
    \[ \max\{ |gA|, |A^{(g)}| \} \gg_{g, \varepsilon} |A|^{g - \varepsilon}. \]
\end{conjecture}

A significant body of work addresses this conjecture, mostly focusing on the case when $g=2$. Some of the highlights in this setting have been the works of Elekes \cite{El1997} and Solymosi \cite{So2009}, who used geometric insights to provide short proofs of strong sum-product estimates. The current best known result here arises from some very recent work of Cushman \cite{Cu2025}, who employed incidence geometric and additive combinatorial methods to prove that
\[ \max\{|2A|, |A^{(2)}| \} \gg_{\varepsilon} |A|^{\frac{4}{3} + \frac{10}{4407} - \varepsilon } \]
for all sets $A \subseteq \mathbb{R}$ and all $\varepsilon>0$.

One can interpret the sum-product phenomenon as a disruption of structure between two non-isogenous algebraic groups. Indeed, one can set $\mathbb{G}_m = (\mathbb{C}^\ast, \times)$ and $\mathbb{G}_a = (\mathbb{C}, +)$ and take $\mathcal{C}$ to be the \emph{correspondence} whose complex points are of the form $(x, x)$ for $x \in \mathbb{G}_m(\mathbb{C})$. See \S3 for a brief introduction about algebraic groups and correspondences between them. We let $\pi_1 : \mathcal{C} \to \G_m$ and $\pi_2:\mathcal{C}\to \G_a$ be the standard projection maps. Given a finite set $A \subseteq \G_m(\mathbb{C})$, we write
\[\mathcal{C}(A) = \bigcup_{x \in A} \pi_2(\pi_1^{-1}(x)) .\]
Note that $\mathcal{C}(A)$ is a subset of $\G_a(\mathbb{C})$. The sum-product phenomenon is now equivalent to saying that for any finite $A \subseteq \mathbb{G}_m(\mathbb{C})$, either $|A+A|$ or $|\mathcal{C}(A) +\mathcal{C}(A)|$ must be much larger than $|A|$.

In a very nice paper,  Bays--Breuillard \cite{BB2021} employed a model theoretic approach to generalise this circle of ideas to a much more broader family of algebraic groups. In particular, given $1$-dimensional, connected, non-isogenous algebraic groups $G$ and $H$ over $\C$ and given some algebraic correspondence $\mathcal{C}$ between $G$ and $H$ of degree $d$, Bays--Breuillard\footnote{Bays--Breuillard  \cite{BB2021} actually proved that for non-constant rational maps $f_1 : G \to \mathbb{C}$ and $f_2 : H \to \mathbb{C}$, and all finite sets $A \subseteq \C$, one either has $|f_1^{-1}(A) + f_1^{-1}(A)| \geq c |A|^{1 + \delta}$ or $|f_2^{-1}(A) + f_2^{-1}(A)| \geq c|A|^{1 + \delta}$, where $c,\delta>0$  are constants depending on $G,H, f_1$ and $f_2$. This can be written in the framework of correspondences by considering the correspondence given by an irreducible component of $\{(x,y) \in G \times H : f_1(x) = f_2(y)\}$.} proved that there exists some $\delta = \delta(G,H,\mathcal{C}) >0 $ such that for any finite set $A \subseteq G$, one has
\begin{equation} \label{bb19}
\max\{ |A+A| , |\mathcal{C}(A) +\mathcal{C}(A)| \} \gg_{G,H,C} |A|^{1 + \delta}. 
\end{equation}

In contrast to the $g=2$ setting of Conjecture \ref{erdszem}, much less is known about the case when one requires \emph{unbounded expansion}, that is, given any real number $k >1$, one wishes to find some $1 \leq g \ll_k 1$ such that
\begin{equation} \label{unbexp}
\max\{ |gA|, |A^{(g)}| \} \gg_k |A|^k 
\end{equation}
holds for all finite sets $A \subseteq \mathbb{C}$. While incidence geometric methods seem to work quite effectively when $k \leq 2$, they do not seem to give any results when $k > 2$. In a breakthrough paper, Bourgain--Chang \cite{BC2004} employed intricate harmonic analytic techniques along with a clever arithmetic lemma to prove that \eqref{unbexp} holds for all finite sets $A \subseteq \mathbb{Q}$. Their work was subsequently simplified and quantitatively improved by P\'{a}lv\"{o}lgyi--Zhelezov \cite{PZ2021}, giving the best known bounds for $g$ in terms of $k$ for this problem. These ideas have since been extended by replacing the sumset $gA$ by other measures of additive structure. In particular, Hanson--Roche-Newton--Zhelezov \cite{HRNZ2020} proved an analogue of \eqref{unbexp} with the sumset $gA$ replaced by the shifted product set $(A+1)^{(g)}$. This has been further generalised by the second author \cite{Mu2024b}, who proved analogues of \eqref{unbexp} with $gA$ first replaced by the sumset $\varphi_1(A) + \dots + \varphi_g(A)$ and then by the product set $\varphi_1(A) \dots \varphi_g(A)$, for suitably chosen polynomials $\varphi_1, \dots, \varphi_g \in \mathbb{Z}[x]$ with bounded degree. 

All the aforementioned results in \cite{BC2004, PZ2021, HRNZ2020, Mu2024b} crucially employed properties about prime factorisation of integers, and subsequently do not generalise to sets of real numbers. In fact, an important problem in this area was to prove \eqref{unbexp} for sets $A \subseteq \mathbb{R}$.  Building on earlier work of Chang \cite{Ch2009}, the second author \cite{Mu2024a} utilised results from diophantine geometry to prove this conditionally on an infamous conjecture in additive combinatorics known as the weak polynomial Freiman--Ruzsa conjecture over $\mathbb{Z}$. The latter has now been resolved in the spectacular work of Gowers--Green--Manners--Tao \cite{GGMT2025}.

Our main result on the generalised sum-product phenomenon is a Bourgain--Chang type unbounded expansion result in the vastly broader setting of 1-dimensional, connected algebraic groups over $\mathbb{C}$.

\begin{thm}\label{thmdelta}  Given integers $k \geq 1$ and $d \geq 2$, there exists an integer $g \geq 1$ such that the following holds. Let $G$ and $H$ be algebraic groups of dimension $1$, and let $\mathcal{C}_1, \cdots, \mathcal{C}_g$ be correspondences  of degree $d$ between $G$ and $H$. Suppose no $\mathcal{C}_i$ is a translate of an algebraic subgroup, and suppose that $G$ is not isomorphic to $\mathbb{G}_a$.  Then for all finite, non-empty sets $A \subseteq G$, one has
$$\max\{|gA|, |\mathcal{C}_1(A) + \dots + \mathcal{C}_g(A)|\} \gg_{d,k}|A|^{ k}. $$
\end{thm}

Setting $G = \mathbb{G}_m$ and $H = \mathbb{G}_a$, we let $\mathcal{C}_i$ be given by the graph of the inclusion $\mathbb{C}^\ast \hookrightarrow \mathbb{C}$. The conclusion of Theorem \ref{thmdelta} then immediately implies \eqref{unbexp} for arbitrary finite sets $A \subseteq \mathbb{C}$. Similarly setting $G = \mathbb{G}_m$ and $H = \mathbb{G}_m$ or $\mathbb{G}_a$ and setting $\mathcal{C}_i$ to be given by the graph of $(x, \varphi_i(x))$ in $G \times H$ for suitably chosen polynomials $\varphi_1, \dots, \varphi_g$ delivers the following corollary. 
\begin{corollary} \label{shiftedbc}
    For all integers $k \geq 2$, there exists an integer $g \geq 2$ such that the following holds. For any non-constant $\varphi_1, \dots, \varphi_g \in \mathbb{C}[x]$ with degree at most $d$ and any finite set $A \subseteq \mathbb{C}$, one has
    \[ \max\{ |A^{(g)}| , |\varphi_1(A) + \dots + \varphi_g(A)| \} \gg_{d,k} |A|^k. \]
    Moreover, if each $\varphi_i(x)$ for $1 \leq i \leq g$ is not of the form $a x^n$ for any $a \in \mathbb{C}$ and any $n \in \mathbb{N}$, then we also have
    \[  \max\{ |A^{(g)}| , |\varphi_1(A) \dots  \varphi_g(A)| \} \gg_{d,k} |A|^k. \]
\end{corollary}

This recovers the results of \cite{HRNZ2020, Mu2024b} qualitatively in the much more general setting where $A \subseteq \mathbb{C}$ instead of $A \subseteq \mathbb{Q}$ and the polynomials $\varphi_1, \dots, \varphi_g$ are allowed to have complex coefficients instead of rational coefficients. We defer further applications of our methods, including an alternative proof of a conjecture of Bays--Breuillard, to \S \ref{furtherapp}. 

We also consider problems concerning expansion in the image set of polynomials as well as intersection of varieties with discrete boxes in algebraic groups. This relates to Elekes--Szab\'{o} and Elekes--R\'{o}nyai type problems. We briefly describe the former, and so, given  $n \geq 3$ and some polynomial $P \in \mathbb{C}[x_1, \dots, x_n]$, when is it the case that for any finite, non-empty set $A \subseteq \mathbb{C}$, one has
\begin{equation}  \label{powsave}
|\{(a_1, \dots, a_n) \in A^n : P(a_1, \dots, a_n) = 0 \}| \ll |A|^{n-1 - \eta}
\end{equation}
for some constant $\eta>0$? Such a characterisation was first studied  by  Elekes--Szab\'{o} \cite{ES2012}, and since then, has seen a flurry of activity, in part due to its connections to a variety of combinatorial geometric problems \cite{RSS2016, RSdZ2016, SZ2024} as well as to model theoretic results \cite{CS2021, BB2021, CPS2024}.

In their aforementioned work, Bays--Breuillard \cite{BB2021} introduced a model theoretic approach to this question, thus generalising the above results for irreducible algebraic sets in $\mathbb{C}^n$. They further noted that upon restricting to a special family of sets $A \subseteq \mathbb{C}$, one can obtain power saving of the shape \eqref{powsave} for a much broader collection of varieties. In particular, in the setting of a $1$-dimensional, complex, connected algebraic group $G$, they proved that for any subvariety $\mathcal{V}\subseteq G^n$ which is not a coset of a subgroup, there exist constants $\varepsilon, \eta>0$ depending only on $G$ and $\mathcal{V}$ such that for any finite set $A  \subseteq G$ satisfying $|A+A| \leq |A|^{1 + \varepsilon}$, one has 
\begin{equation} \label{bsp}
|A^n \cap \mathcal{V}| \ll_{G,\mathcal{V}} |A|^{\dim(\mathcal{V}) - \eta}. 
\end{equation}

It is natural to ask what is the best possible value of $\eta$ that is admissible in \eqref{bsp}. When $G$ is not isomorphic to $\mathbb{G}_a$, we resolve this question.

\begin{theorem} \label{esz}
Let $G$ be a connected algebraic group over $\mathbb{C}$ of dimension $1$. Suppose that $G$ is not isomorphic to $\mathbb{G}_a$. Let $A \subseteq G$ be a finite set such that $|A+ A| \leq K|A|$ for some $K \geq 1$. Then for any irreducible subvariety $\mathcal{V}\subseteq G^g$, that is not a translate of an algebraic subgroup of $G^g$, one has
   \begin{align*}
       |\mathcal{V}\cap A^g| \ll_{g, \deg(\mathcal{V})} K^{C } + |A|^{\dim(\mathcal{V}) - 1},
   \end{align*} 
   where $C>0$ is some constant depending only on $\deg(\mathcal{V})$ and $g$. 
\end{theorem}

In particular, when $G$ is not isomorphic to $\mathbb{G}_a$ and $\mathcal{V}$ is not a translate of an algebraic subgroup, there exists some $\varepsilon >0$ depending only on $\mathcal{V}$ and $g$,  such that whenever $|A+A| \leq |A|^{1  + \varepsilon}$, one has 
\[ |\mathcal{V}\cap A^g| \ll_{g, \deg(\mathcal{V})} |A|^{\dim(\mathcal{V}) - 1}. \]
We further note that this upper bound is of the right order.

\begin{example} \label{cosetexample}
    Let $G = \G_m$ and $g = 3$. The variety $\mathcal{V}$ defined by the polynomial 
\begin{equation}  \label{degenex}
P(X_1, X_2, X_3) = X_2X_3 - X_1 + 1 
\end{equation}
contains the translates $X_2X_3 = \gamma, X_1 = \gamma + 1$ of the algebraic subgroup $X_2 X_3 = 1, X_1 = 1$, for any $\gamma \not \in \{-1, 0\}$. Setting 
\[ A = \{\gamma^i : -N \leq i \leq N\}\cup\{\gamma + 1\},\]
it is easy to see that for all $N \geq 3$, we have $|\mathcal{V}\cap A^3| \gg |A|$ and $|A\cdot A| \leq 3|A|$. However, $\mathcal{V}$ is not a translate of an algebraic subgroup.  
\end{example}

In fact, our methods deliver even stronger upper bounds that depend on the largest dimension of a maximal translate of an algebraic subgroup contained in $\mathcal{V}$, see Theorem \ref{elsz} for further details.

We also note that while previous approaches to Elekes--Szab\'{o} type problems employed combinatorial geometric methods or model theory, our own approach utilises a novel interaction between Mordell--Lang and $S$-unit type results along with recent developments in additive combinatorics involving Freiman--Ruzsa type results.

One can similarly consider expansion in the image set of polynomials, that is, given some polynomial $P \in \mathbb{C}[x_1, \dots, x_n]$ and some finite set $A \subseteq \mathbb{C}$, when is the set 
\[ P(A, \dots, A) = \{ P(a_1, \dots, a_n) : a_1, \dots, a_n \in A\}\]
significantly larger than $A$. The first result in this direction is due to Elekes--R\'{o}nyai \cite{ER2000} who proved that either $P \in \mathbb{C}[x,y]$ is degenerate, in the sense that $P = h(f(x) + g(y))$ or $P = h(f(x)g(y))$ for some univariate polynomials $f,g,h$, or one has $|P(A,A)| \gg |A|^{1 + \eta}$ for every finite $A \subseteq \mathbb{C}$, with $\eta>0$ being an absolute constant. There have been various subsequent quantitative improvements, as well as explorations of cases where, as before, one restricts to a special family of sets $A$ in order to widen the choices for $P$ and get better quantitative values of $\eta$, see \cite{BB2021, JRT2022, Mu2024a} as well as \cite{HRNZ2020, Mu2024b} for related sum-product type problems.

Our main result in this direction is a significant generalisation of the above result in the wider setting of algebraic groups. In order to state this, we will need the following definition.

\begin{definition} \label{degenvar}
Let $G$  be a $1$-dimensional, connected algebraic group over $\mathbb{C}$, and let $\mathcal{B}$ be a projective variety of positive dimension. Let $\pi_1 : G^g \times \mathcal{B} \to G^g$ and $\pi_2: G^g \times \mathcal{B} \to \mathcal{B}$ be the canonical projection maps.  We call an irreducible subvariety $\mathcal{V} \subseteq  G^g \times \mathcal{B}$ \emph{degenerate}, if there exists a connected algebraic group $H \subseteq G^g$ of positive dimension, and a proper subvariety $\mathcal{W} \subsetneq G^g/H\times \mathcal{B} $ such that 
$$\mathcal{V} = \pi_H^{-1}(\mathcal{W}), $$
for the projection $\pi_H: G^g \times \mathcal{B} \rightarrow G^g/H\times \mathcal{B}$. 
    If $\mathcal{V}$ is not degenerate, we call $\mathcal{V}$ \textit{non-degenerate}. 
\end{definition} 

An example of a degenerate subvariety is given by the equation $P(X_1, X_2, X_3) = t$, where $P$ is given by \eqref{degenex}. An example of a non-degenerate subvariety is provided after the statement of Theorem \ref{elekesronyai} below.

With this in hand, we state our result as follows.

 \begin{theorem} \label{elekesronyai}
Let $G, \mathcal{B}, \pi_1$ and $\pi_2$ be as in Definition \ref{degenvar}, with $G$ not isomorphic to $\G_a$. Let $\mathcal{V} \subseteq G^g\times \mathcal{B}$ be non-degenerate of dimension $g$ and degree $d$, and such that $\pi_1$ and $\pi_2$ restricted to $\mathcal{V}$ are dominant. Let $A \subseteq G$ be a finite, non-empty set such that $|A+A| \leq K|A|$. Then for all $X \subseteq A$, we have
\begin{equation} \label{elrexp}
    |\pi_2((X^g\times \mathcal{B}) \cap \mathcal{V})| \gg_{d,g} \frac{|X|^g}{K^{O_{d,g}(1)}}.
\end{equation}
 \end{theorem}

As an example, one may set $G = \mathbb{G}_m, B = \mathbb{P}^1$ and the variety $\mathcal{V}$ to be defined by the equation $P(x_1, \dots, x_g) = t$, where $P \in \mathbb{C}[x_1, \dots, x_g]$ is some polynomial. We refer to $P$ as \emph{non-degenerate} with respect to $G^g$ if the variety $\mathcal{V}$ is non-degenerate. One can see that this is equivalent to the fact that $P(\vec{x}) \neq F(m_1(\vec{x}), \dots, m_{g-1}(\vec{x}))$ for any choice of monomials $m_1, \dots, m_{g-1} \in \mathbb{C}[x_1, \dots, x_g]$ and any $F \in \mathbb{C}[y_1, \dots, y_{g-1}]$. In this case, we have
\[ \pi_2((A^g\times \mathcal{B}) \cap \mathcal{V}) = P(A,\dots, A) . \] 
Thus \eqref{elrexp} implies that for any finite set $A \subseteq \mathbb{C}^{*}$ satisfying $|A \cdot A| \leq K|A|$, one has
\begin{equation}  \label{fewprd}
|P(A,\dots, A)| \gg_{d,g} \frac{|A|^g}{K^{O_{d,g}(1)}}. 
\end{equation}
The lower bound in \eqref{fewprd} was proved in \cite{Mu2024a} conditional on the weak PFR conjecture over $\mathbb{Z}$. In fact, Theorem \ref{elekesronyai} can be seen as a generalisation of the results from \cite{Mu2024a} to the much broader setting of varieties in $1$-dimensional algebraic groups over $\mathbb{C}$.

Apart from having applications to various sum-product type questions, a nice aspect of this notion of degeneracy is that it is optimal. In particular, if a polynomial $P$ is degenerate in the above sense, then \cite[Proposition 1.2]{Mu2024a} implies that for any finite $A \subsetneq \mathbb{Z}$ with $|A \cdot A| \leq K |A|$ one has 
\[ |P(A, \dots, A)| \ll_P K^{O_P(1)}|A|^{g-1} .\]
Furthermore, the lower bound in \eqref{elrexp} is almost optimal in the sense that it matches the trivial upper bound $|X|^g$ up to factors of $K^{O_{d,g}(1)}$.

We note that in some of our results, including Theorem \ref{elekesronyai}, we assume that the algebraic group $G$ is not isomorphic to $\G_a$. In fact, this is a necessary condition for many of these to hold. For example, Theorem \ref{elekesronyai} is not true for the case when $G = \G_a$. 

\begin{example}
    Let $P(x,y,z) = xy + yz + zx$. One can show that this polynomial is non-degenerate with respect to $\G_a^3$, see Appendix \ref{Appb}. Moreover, the set $A = \{1,2,\dots, N\}$ is a subset of $\G_a$ with $|A+A| \leq 2|A|$. Finally, $|P(A,A,A)| \ll N^2 = |A|^2$ implying that a conclusion akin to \eqref{fewprd} in this case fails to hold true. 
\end{example}

It is worth mentioning that Theorem \ref{elekesronyai} can be employed to prove its counterpart where we replace the condition that $A$ has a small sumset with $A$ lying in a subgroup of small rank, see Theorem \ref{projections}. In particular, given some subgroup $\Gamma \subseteq G$ of rank $r$ and some finite set $Y \subseteq \Gamma$, Theorem \ref{elekesronyai} implies that one must have
\begin{equation} \label{lowrank}   
|\pi_2((Y^g\times \mathcal{B}) \cap \mathcal{V})| \gg_{d,g} \frac{|Y|^g}{2^{O_{d,g}(r)}}. 
\end{equation}
Indeed, let $\Gamma$ be generated by $\gamma_1, \dots, \gamma_r$. Now, given any finite set $Y \subseteq \Gamma$, we can find some $L \in \mathbb{N}$ such that $Y \subseteq A$, where $A = \{ n_1 \gamma_1 + \dots + n_r \gamma_r : |n_1|, \dots, |n_r| \leq L\}$. Moreover, note that $|A+A| \leq 2^r |A|$. We may now apply Theorem \ref{elekesronyai} to obtain \eqref{lowrank}. Furthermore, this means that we may deduce Theorem \ref{thmBremner} and its generalisation Corollary \ref{corBremner} via a combination of inequality \eqref{lowrank}  and Proposition \ref{nondegenerate}.

An important step towards proving our results is that the above implication can be roughly reversed as well. We perform this reversal by combining the recent resolution of the weak polynomial Freiman--Ruzsa conjecture due to Gowers--Green--Manners--Tao \cite{GGMT2025} along with various additive combinatorial techniques and the fact that a finite subgroup of an algebraic group has a uniformly bounded number of generators, which is a simple consequence of its Lie theory.  

Thus, it suffices to work in the setting where our sets are lying in subgroups of bounded rank. One of our key results here is Theorem \ref{projections}, whose conclusion is also recorded in  \eqref{lowrank}. This is where a significant portion of our input from Diophantine geometry comes in, including utilisation of a uniform version of Mordell--Lang by David--Philippon \cite{DavidPhilippon} and the $S$-unit bounds by Evertse--Schlickewei--Schmidt \cite{ESSlinear}. 

A third step that is required to prove our sum-product results such as Theorem \ref{thmdelta}, involves showing that an auxiliary variety, which captures the movement of additive structure through the correspondences, is non-degenerate in the sense of Definition \ref{degenvar}. This is precisely the content of Proposition \ref{nondegenerate}. For the proof, which can be found in section \ref{sectioncorr}, we work on the tangent space of our algebraic groups. We note that Proposition \ref{nondegenerate}, when combined with Theorem \ref{projections}, implies a suitable variant of Theorem \ref{thmdelta} which holds for sets contained in finite rank subgroups, see Theorem \ref{rankversion}. The latter is already sufficient for our applications to Bremner's conjecture.

  There is little doubt in the authors minds that the slick interaction of Diophantine geometry with additive combinatorics that is apparent here seems to suggest that correspondences between algebraic groups present a very suitable framework to conceptualise the sum-product phenomenon. One way of viewing this interaction is that  the weak polynomial Freiman--Ruzsa conjecture over $\mathbb{Z}$ is a statement concerning simply addition in $\mathbb{Z}^r$. The latter can be embedded into algebraic groups via rank $r$ groups in a myriad ways. This is combined with the algebraic structure that is implicitly present in $\G_a, \G_m$ and elliptic curves. Ultimately, the Mordell--Lang conjecture tells us how the group arithmetic interacts with the Zariski-topology of the groups. This is especially convincing, if we remember the special role played by isogenies as these are precisely the maps that respect both the algebraic and the group structure. Any correspondence that is not a translate of an algebraic subgroup should destroy the approximate group structure of any finite set as it transports the set from one group law to another.

\subsection*{Outline}
We will present some further applications of our ideas in \S\ref{furtherapp}. We use \S\ref{setupsec} to give a brief introduction about algebraic groups and correspondences, as well as record some consequences of the uniform version of Mordell--Lang by David--Philippon \cite{DavidPhilippon} and the $S$-unit bounds by Evertse--Schlickewei--Schmidt \cite{ESSlinear}. We dedicate \S\ref{pcp} to proving Theorem \ref{projections}, and in \S\ref{sectioncorr}, we will prove Proposition \ref{nondegenerate}. We use \S\ref{addcom} to prove the additive combinatorial structural results that we require for the proofs of our results. Finally in \S\ref{alltheproofs}, we provide all the proofs of our results from \S\ref{introduction} and \S\ref{furtherapp}.
In  Appendix \ref{appa}, we give some applications of our results to Diophantine equations. Moreover, we make some brief remarks about properties of degenerate polynomials in Appendix \ref{Appb}.

\subsection*{Notation} We use Vinogradov notation. Thus we write $X \ll_z Y$ to mean that $|X| \leq C Y$ where $C>0$ is some constant depending on the parameter $z$. We write $X = O(Y)$ to mean $X \ll Y$, and we write $X\asymp Y$ to mean $X \ll Y \ll X$. We will often write $A \subseteq G$  for a finite set $A$ and an algebraic group $G$. In this case we identify (by abuse of notation) $A$ with a 0-dimensional algebraic subvariety consisting of the points of $A$ with multiplicity 1. 

\subsection*{Acknowledgements}  The second author is supported by a Leverhulme early career fellowship \texttt{ECF-2025-148}.

\section{Further Applications} \label{furtherapp}

\subsection{Generalised Bremner}
We mention a generalised version of Theorem \ref{thmBremner} for correspondences between algebraic groups. We first record an expansion version for correspondences.  

\begin{thm} \label{rankversion}
For all integers $d \geq 1$ and $g \geq 2$, there is a positive constant $0 < c(d,g) <1$ with the following property.
Let $\mathcal{C}_1, \cdots, \mathcal{C}_g$ be correspondences  of degree $d$ between algebraic groups $G$ and $H$ of dimension 1. Suppose no $\mathcal{C}_i$ is a translate of an algebraic subgroup and that $G$ is not isomorphic to the additive group $\mathbb{G}_a$. Let $\Gamma \subseteq G(\mathbb{C}) $ be a subgroup of finite rank $r$. Then for any finite subset $A\subseteq \Gamma$, one has
$$|\mathcal{C}_1(A) + \cdots + \mathcal{C}_g(A)|\geq c(d,g)^{1 + r} |A|^{g}.$$
\end{thm}
 We note that the theorem is slightly asymmetric as we can not allow $G$ to be isomorphic to $\mathbb{G}_a$. This is indeed necessary as for example $\Gamma = \mathbb{Z}$, $H = \mathbb{G}_m$ and $\Delta$ the diagonal as described above shows that if we drop that assumption that would imply that any set in $\mathbb{Z}$ has big product set, which is wrong.

It is also worth noting that fixing an elliptic curve $E$ over a number field $K$, there are only finitely many elliptic curves over $K$, that are isogenous to it, even over an algebraic closure. This is a consequence of Faltings's famous theorem \cite{Faltings}, later significantly improved by Masser--Wüstholz \cite{MWelliptic}, which shows that, even for a fixed number field, our theorems apply to a vast zoo of non-isogenous algebraic  groups. 

Given $k \in \mathbb{N}$, we define a \emph{proper generalised arithmetic progression of rank $k$} to be a set $P$ of the form 
\begin{equation} \label{gapdef}
P =\{P_0 + \ell_1P_1 + \cdots +\ell_kP_k \  :  \ 0\leq \ell_i \leq L_i-1\} , 
\end{equation}
where $P_0, \dots, P_k \in H$ and $L_1, \dots, L_k \geq 2$ are integers and one has $|P| = L_1 \dots L_k$. These sets play a crucial role in additive combinatorics and number theory since they act as an important family of sets that exhibit additive structure. With this in hand, we now prove our main result on Bremner's conjecture and related questions.

\begin{cor}\label{corBremner} For all integers $d \geq 1$ there exists a constant $D= D(d)$ with the following property. Let $G$ be either $\mathbb{G}_m$ or an elliptic curve $E$, and let $\mathcal{C}$ be a correspondence of degree at most $d$ between $G$ and an algebraic group $H$ of dimension 1, that is not the translate of an algebraic subgroup. Then for any subgroup  $\Gamma \subseteq G(\C)$ of rank $r$, a proper generalised arithmetic progression $P$ of rank $k$ in $\mathcal{C}(\Gamma)$ satisfies 
$$|P| \leq D^{1 + r}. $$
\end{cor}

 Theorem \ref{thmBremner} follows in a straightforward manner from the above result, see \S\ref{alltheproofs}. Corollary \ref{corBremner} also gives a more general and uniform version of \cite[Theorem 6.1]{GFP2021}.   
 
A nice aspect of our upper bound is that it is completely independent of the rank $k$ of the progression. Moreover, generalised arithmetic progressions are indeed a strictly more general pattern. For example, the generalised arithmetic progression $P'$ as described in \eqref{gapdef} can not be covered by fewer than $C^k$ arithmetic progressions, for some constant $C>1$, but we still obtain a uniform upper bound of the form $|P'| \leq D^{1 + r}$ for some $0 < D \ll_d 1$ which is independent of $k$.

\subsection{Sum-product phenomenon}

Returning to the generalised sum-product phenomenon,  Bays--Breuillard \cite{BB2021} speculated that the exponent $\delta$ in their result recorded in \eqref{bb19} should be independent of $G$ and $H$. Significantly generalising this model-theoretic and incidence-geometric framework, Chernikov--Peterzil--Starchenko \cite{CPS2024} confirmed the speculation of Bays--Breuillard in a quantitative sense. In particular, they proved that $\delta = 1/21$ is admissible in  \eqref{bb19}. As an application of our methods, we prove an asymmetric, uniform version of \eqref{bb19}, thus confirming the speculation of Bays--Breuillard via a very different set of techniques.

\begin{thm} \label{sumprod1}
Given $d \in \mathbb{N}$, there exists some constant $D= D(d)>0$ such that the following is true.  Let $\delta >0$, let $G$ and $H$ be connected algebraic groups of dimension 1 with $G$ not isomorphic to $\mathbb{G}_a$, and let $\mathcal{C}$ be an algebraic correspondence between $G$ and $H$ of degree $d \geq 2$ that is not the translate of an algebraic subgroup.   Then for any finite, non-empty set $A \subseteq G$,  one has either
\begin{equation} \label{fpmsg}
|A+A| > |A|^{1 + \delta} \ \ \text{or} \ \ |\mathcal{C}(A) + \mathcal{C}(A)| \geq D^{-1} |A|^{2 - D\delta }.
\end{equation}
\end{thm}

Indeed, Theorem \ref{sumprod1} implies that $\delta = 1/(D+1)$ is admissible in \eqref{bb19}. While this choice of $\delta$ is much smaller than $1/21$, the main novelty of our result lies in the asymmetry between the lower bounds for $|A+A|$ and $|\mathcal{C}(A) + \mathcal{C}(A)|$ in \eqref{fpmsg}. For instance, if we set $G = \mathbb{G}_m$ and $H = \mathbb{G}_a$ or $\mathbb{G}_m$ and the correspondence $\mathcal{C}$ to be given by the graph of $y = \varphi(x)$ for some suitable polynomial $\varphi \in \mathbb{C}[x]$, we obtain the following corollary.

\begin{corollary} 
    Let $A \subseteq \mathbb{C}$ be a finite set, let $d \geq 1$ be an integer, let $K \geq 1$ and let $\varphi \in \mathbb{C}[x]$ have $\deg \varphi = d$. If $|A \cdot A| \leq K|A|$, then 
    \begin{equation} \label{wesz}
    |\varphi(A) + \varphi(A)| \gg_d |A|^{2}/K^{D}
    \end{equation}
    where $D>0$ is some constant depending on $d$. Moreover, if $\varphi(x)$ is not of the form $c x^d$ for any $c \in \mathbb{C}$, then we also have
    \[ |\varphi(A) \cdot \varphi(A)| \gg_d |A|^2/K^D.\]
\end{corollary}

We note that simply setting $\varphi = x$ in \eqref{wesz} immediately delivers the so-called weak Erd\H{o}s--Szemer\'{e}di Conjecture over $\mathbb{C}$. This was first proven by Bourgain--Chang \cite{BC2004} for sets $A \subseteq \mathbb{Q}$, with work of Chang \cite{Ch2009} delivering this conclusion for sets $A \subseteq \mathbb{R}$, conditional on the weak PFR conjecture over $\mathbb{Z}$. Building on the work of Chang and employing the resolution of the weak PFR conjecture over $\mathbb{Z}$ due to Gowers--Green--Manners--Tao \cite{GGMT2025}, the second author proved that for any finite set $A \subsetneq \mathbb{C}$ with $|A \cdot A| \leq K|A|$, one has at most $|A|^2/K^{O(1)} $ many quadruples $a_1, \dots, a_4 \in A$ such that $a_1 + a_2 = a_3 + a_4$, see \cite[Proposition 1.5]{Mu2024c}. This then immediately implies that $|A+A| \geq |A|^2/K^{O(1)}$.

\subsection{Elekes--Szab\'{o}}

As remarked in \S1, we are able to prove a more general upper bound for quantities of the form $|\mathcal{V}\cap A^g|$ which depend on the maximal dimension of a translate of an algebraic subgroup contained in $\mathcal{V}$. In order to elaborate on this, we present the following definition.

\begin{definition}
    For an irreducible  subvariety $\mathcal{V} \subseteq G^g$, we define the \textit{coset defect}, denoted $\text{codef}(\mathcal{V})$, to be the maximal dimension of a connected algebraic group $H \subseteq G^g$, such that $\gamma + H \subseteq \mathcal{V}$ for some $\gamma \in G^g(\C)$.
\end{definition}

With this in hand, we state our result.

\begin{theorem} \label{elsz}
    Let $G$ be a $1$-dimensional, connected algebraic group over $\mathbb{C}$ not isomorphic to $\mathbb{G}_a$. Let $A \subseteq G$ be a finite set such that $|A + A| \leq K|A|$ for some $K \geq 1$. Then for any irreducible subvariety $\mathcal{V} \subseteq G^g$, one has
   \begin{align*}
       |\mathcal{V}\cap A^g| \ll_{g, \deg(\mathcal{V})} (K^{C} + |A|^{{\rm codef}(\mathcal{V})}),
   \end{align*} 
   where the constant $C>0$ depends only on $\deg(\mathcal{V})$ and $g$. 
\end{theorem} 

We note that $\mathcal{V}$ might be covered by translates of algebraic subgroups even though $\mathcal{V}$ is not a translate of an algebraic subgroup. However, Theorem \ref{elsz} still applies to such varieties. For example, for $g = 3$ the variety $\mathcal{V}$ defined by $X_1X_2^2X_3 - X_2X_3 - 1 =0$ is not a coset, but is covered by cosets of the form $X_2X_3 = \gamma, X_1X_2 = (1 + \gamma)/\gamma $ for $\gamma \in \C \setminus \{-1, 0\}$. Theorem \ref{elsz} admits the following corollary.

\begin{cor}
   For every $d,g \geq 1$ there exists $\epsilon > 0$, such that for an irreducible variety $\mathcal{V}\subseteq G^g$ of degree at most $d$, if $|A + A|\leq |A|^{1 + \epsilon}$ and ${\rm codef}(\mathcal{V}) \geq 1$, then
   \begin{align}\label{ESsharp} |A^g \cap \mathcal{V}|\ll_{g, \deg(\mathcal{V})}|A|^{{\rm codef}(\mathcal{V})}.
   \end{align}
   Finally, if $\mathcal{V}$ does not contain a positive dimensional coset, then for every $\epsilon > 0$, if $|A + A|\leq |A|^{1 + \epsilon}$, then 
   $$|\mathcal{V}\cap A^g|\ll_{g,\deg(\mathcal{V})}|A|^{C\epsilon}, $$
   where $C>0$ is some constant depending on $\deg(\mathcal{V})$ and $g$.
\end{cor}

As in the case of Theorem \ref{esz} and Example \ref{cosetexample}, one can show that the upper bounds in Theorem \ref{elsz} and inequality \ref{ESsharp} are of the right order.

\section{Setup} \label{setupsec}

\subsection{Algebraic groups}

We will be working with connected algebraic groups over $\C$ of dimension $1$. An algebraic group (over $\C$) is an algebraic variety $G$ with a morphism from $G \times G $ to $G$ that induces a group operation on $G(\C)$. As an example, consider the algebraic group given by the variety $\mathbb{A}^1$ along with the morphism  which maps $(x,y)$ to $x+y$. We refer to this algebraic group as the additive group $\mathbb{G}_a$. Another example is the algebraic group given by the algebraic variety $\mathbb{A}^1 \setminus \{0\}$ with the morphism that maps $(x,y)$ to $xy$. We refer to this algebraic group as the multiplicative group $\G_m$. A third example of this is an elliptic curve over $\mathbb{C}$ with its canonical group operation \cite[III.2]{Silverman_the_arithmetic_of_elliptic_curves}.

All the above three examples are $1$-dimensional, connected algebraic groups over $\mathbb{C}$, and in fact, these are essentially the only possible examples. In order to see this, note that the analytification of $G$ is a complex Lie group, and therefore has an exponential map
\begin{align*}
    \exp_G : \C \to G(\C),
\end{align*}
which is analytic and non-constant, since it is a local diffeomorphism \cite[Proposition 20.8 (f)]{Lee_introduction_to_smooth_manifolds}. When $G$ is commutative, $\exp_G$ is a morphism of Lie groups. For this, see \cite[Exercise 20-8]{Lee_introduction_to_smooth_manifolds} or the Baker--Campbell--Hausdorff formula. The following argument in complex analysis now implies that $\exp_G$ is surjective with discrete kernel.

\begin{proposition}
\label{structure_of_morphisms_of_1_dim_Lie_groups}
    Let $G$ and $H$ be complex Lie groups of dimension $1$. Suppose that $H$ is connected. Every morphism of complex Lie groups from $G$ to $H$ is either trivial or is surjective with discrete kernel.
\end{proposition}
\begin{proof}
    Suppose the morphism $f$ is not trivial. Discreteness of the kernel follows from the uniqueness of analytic continuation. Since $\dim(G) = \dim(H) = 1$, the open mapping theorem implies that the image $U$ of $f$ is open. The set $U$ is also a Lie subgroup of $H$. Therefore $H$ is a disjoint union of the cosets of $U$, each of which is open. Since $H$ is connected, there can be at most one coset, i.e., $U = H$. 
\end{proof}

There are therefore three options for $G$ depending on the kernel of $\exp_G$.
\begin{enumerate}[(1)]
    \item The kernel is trivial. In this case $G(\C)$ is isomorphic to the additive group of complex numbers $\C$.
    \item The kernel is a lattice of rank one. In this case $G(\C)$ is isomorphic to the multiplicative group of complex numbers $\C^\ast$. For example, the linear map on the tangent space $\C$ that takes the kernel of $\exp_G$ to the lattice $2\pi i \Z$ is such an isomorphism.
    \item The kernel is a lattice $\Lambda$ of rank two. In this case $G(\C)$ is isomorphic to the complex torus $\C/\Lambda$. It can be shown via the classical Weierstrass theory \cite[Proposition VI.3.6]{Silverman_the_arithmetic_of_elliptic_curves} that $G$ is isomorphic to the group of complex points of an elliptic curve.
\end{enumerate}
In each case, the isomorphism of $G(\C)$ with $\C, \C^\ast$ or $\C/\Lambda$ can be promoted to an isomorphism of $G$ with $\G_a, \G_m$ or an elliptic curve $E$ by extending the isomorphism to the closure of $G(\C)$ in some projective space, and applying Serre's GAGA theorem.

The exponential map $\exp_G$ of an algebraic group admits a local inverse, which we denote by $\log_G$. We note that $\exp_{\G_a(\C)} = \id_{\C}$ and $\exp_{\G_m(\C)}$ is the usual exponential function $\exp : \C \to \C^\ast$. Furthermore, when $G$ is an elliptic curve embedded into $\P^2$ via its Weierstrass form, then $\exp_G : \C \to G(\C) 
    \subseteq \P^2(\C)$  satisfies
\[   \exp_G(z) =(2\sigma(z)^3 \wp(z) : \sigma(z)^3 \wp'(z) : 2\sigma(z)^3) \]
for all $z \in \mathbb{C}$, where $\wp$ and $\sigma$ denote the classical Weierstrass $\wp$-function and $\sigma$-function associated to the lattice given by the kernel of $\exp_G$.

We remark that this classification of connected algebraic groups of dimension $1$ extends to any algebraically closed field in characteristic zero. This follows from the Barsotti--Chevalley--Rosenlicht theorem, see \cite[Theorem 10.25]{Milne_algebraic_groups}.

If $G$ is an algebraic group, then a closed subvariety of $G$ is called an algebraic subgroup if it is an algebraic group with the same group operation. In particular, we require algebraic subgroups to be closed, but not irreducible.

Given algebraic groups $G$ and $H$, a morphism from $G$ to $H$ is a morphism of the underlying varieties that also induces a group homomorphism from $G(\C)$ to $H(\C)$. For example, every morphism from $\G_m$ to $\G_m$ is given by sending $x$ to $x^n$ for some $n \in \Z$. Moreover, every morphism from $\G_m$ to $\G_a$ is trivial; that is, it sends $x$ to $0$. The kernel of a morphism of algebraic groups is an algebraic subgroup. Moreover, if $\dim(G) = \dim(H) = 1$, then any morphism of algebraic groups from $G$ to $H$ is either trivial or it is surjective with finite kernel. This follow from Proposition \ref{structure_of_morphisms_of_1_dim_Lie_groups} upon observing that the kernel is discrete in the analytic topology and closed in the Zariski topology. Morphisms of algebraic groups that are surjective with finite kernel are called \emph{isogenies}. Moreover, $G$ and $H$ are called \emph{isogenous} if there is an isogeny from $G$ to $H$. Thus, any morphism of algebraic groups of dimension $1$ is either trivial or an isogeny. If two connected algebraic groups of dimension $1$ are isogenous, then they are either isomorphic or they are elliptic curves. 
We will use the following well-known fact about complex algebraic groups. 
\begin{lem}\label{lemtangent} Let $G, H$ be one dimensional complex connected algebraic groups and $\mathcal{C} \subsetneq G \times H$ an algebraic correspondence. Suppose that there is a one dimensional vector space $V \subsetneq \mathbb{C}^2, b \in \mathbb{C}^2$ and  a non-empty open set $U \subseteq \mathbb{C}^2$, such that $U \cap (V + b) $ is non-empty and
    $$\exp_{G \times H}(U \cap (V + b)) \subseteq \mathcal{C}(\mathbb{C}).$$
    Then $\mathcal{C}$ is a translate of an algebraic group. 
\end{lem}
We will provide a proof for the benefit of the reader. 
\begin{proof} We can translate $\mathcal{C}$ by $P = \exp_{G \times H}(b)$ so that we can assume that $b = 0$ and $U$ is a neighbourhood of the identity. Now $\exp_{G \times H}(U \cap V) \subseteq \mathcal{C}(\mathbb{C})$. There is an open non-empty set $U' \subseteq U $, such that $U + x\cap U $ is non-empty for all $x \in U'$. Thus $\exp(U') \subsetneq \text{Stab}(\mathcal{C}) = \{P \in \mathcal{C}(\mathbb{C}); P + \mathcal{C} = \mathcal{C}\}$. Since the stabiliser $\text{Stab}$ is an algebraic variety and $\mathcal{C}$ is irreducible $\mathcal{C} = \text{Stab}(\mathcal{C})$. We also note that the stabilizer is a group and thus an algebraic subgroup of $G \times H$.
\end{proof}

\subsection{Degrees}\label{degrees}

Given an algebraic group $G$ and some subvariety $\mathcal{V}$ of $G$, we will often need to define the degree of $\mathcal{V}$. It is worth mentioning that our varieties will always be pure dimensional and otherwise we talk about Zariski-closed sets. In order to define the degree of $\mathcal{V}$, we need a map from from $G$ to projective space $\mathbb{P}^n$ for some $n \in \mathbb{N}$. Such maps are parameterised by line bundles.

Thus, let $G$ be an algebraic group of dimension $1$. If $G$ is $\G_a$ or $\G_m$, we fix a canonical open immersion $G \to \P^1$. In this case, the Zariski closure $\overline{G}$ of $G$ in $\mathbb{P}^1$ satisfies $\overline{G} = \mathbb{P}^1$. We let $L_G$ denote the line bundle $O_{\P^1}(1)$ on $\P^1$. If $G$ is an elliptic curve, let $L_G$ be the ample line bundle $O_G(O)$, where $O$ is the identity of $E$. Moreover, in the case of elliptic curves, we have $\overline{G} = G$.

On a product of algebraic groups $G_1 \times \dots \times G_g$, we will always use the line bundle
\begin{align*}
  L=   (\pi_1^\ast L_{G_1}) \otimes \dots \otimes
    (\pi_g^\ast L_{G_g}),
\end{align*}
where for each $i \in \{1, \dots, g\}$, the map
\begin{align*}
    \pi_i : \overline{G_1} \times
    \dots \times
    \overline{G_g} \to \overline{G_i}
\end{align*}
is the projection morphism, and $\pi_i^\ast$ is the pullback morphism on line bundles. 

With the line bundle $L$ fixed, the \emph{degree} $\deg_L(\mathcal{V})$ of a (quasi-projective) subvariety $\mathcal{V}$ of dimension $n$ in $\overline{G}_1 \times \dots \times \overline{G}_g$ is the intersection product 
\begin{align*}
    \deg_L(\mathcal{V}) = c_1(L)^n \cdot [V],
\end{align*}
where $c_1(L)$ is the first Chern class of $L$. For the definition of the intersection product and Chern classes, see \cite[Chapter 2.5]{Fulton}. A viewpoint requiring less machinery is that a multiple of $L$ is very ample and gives an embedding into projective space. The degree of $\mathcal{V}$ is then the degree of the image of the embedding.


        

\begin{remark} \label{exga}
If $G$ is $\G_a$ or $\G_m$ and $\mathcal{V}$ is a hypersurface defined by a single polynomial, the above definition somewhat closely resembles an intuitive definition of the degree of polynomial. In particular, letting $X_1, \dots X_g$ denote the cartesian coordinates of $G^g$, we view the hypersurface $\mathcal{V}$ as a subvariety of $\mathbb{A}^g$ defined by some polynomial 
\[ f_\mathcal{V}= 
\sum_{\vec{\lambda} \in A}c(\vec{\lambda})X^{\vec{\lambda}},\]
where $A$ is some finite, non-empty subset of $\Z_{\geq 0}^g$, $c(\vec{\lambda}) \neq 0$ for all $\vec{\lambda} \in E$, and $X^{\vec{\lambda}} = X_1^{\lambda_1} \dots X_g^{\lambda_g}$. Then it turns out that $\deg_L(\mathcal{V}) = j_1 + \dots +j_g$, where for each $1 \leq i \leq k$, the number $j_i$ is the largest non-negative integer $j$ such that there is a monomial $X^{\vec{\lambda}}$, with $\vec{\lambda} \in E$, which is divisible by $X_i^{j}$. This can be strictly larger than the usual total degree of a polynomial, which is defined as the largest degree of a monomial with non-zero coefficient.
\end{remark}

\subsection{Correspondences}

We first make the concept of a correspondence precise.
\begin{definition}
    Let $X$ and $Y$ be irreducible curves. A correspondence $\mathcal{C}$ between $X$ and $Y$ is an irreducible curve $\mathcal{C} \subseteq X \times Y$ such that the canonical projections $\pi_X: \mathcal{C} \to X$ and $\pi_Y: \mathcal{C} \to Y$ are dominant. Moreover, for any set $A \subseteq X(\C)$, we define 
    \[ \mathcal{C}(A) = \{ \pi_Y( \pi_X^{-1}(x)) : x \in A\} =  \pi_Y( \mathcal{C} \cap (A \times Y) ).\]
\end{definition}

 Here, we recall that the projection $\pi_X$ is finite if for every $x \in X$, the set $\pi_X^{-1}(x) = \{ z \in \mathcal{C}(\C) : \pi_X(z) = x\}$ is finite. Since all curves involved are irreducible this is equivalent to $\pi_X$ being dominant.  We recall that the projection $\pi_X$ is dominant if $\pi_X(\mathcal{C}(\C))$ is dense in $X$. Note that here we use the fact that $\pi_X(\mathcal{C}(\C))$ is either finite, empty or co-finite; this is not true for an arbitrary dense set $A \subseteq X(\C)$.

If $A \subseteq X(\C)$ is a finite set, then $\mathcal{C}(A)$ is also finite by our assumptions on the dimensions of $X, Y$ and $\mathcal{C}$. If $d_X$ and $d_Y$ are the degrees of the projection maps $\pi_X$ and $\pi_Y$, then since $X,Y,\mathcal{C}$ are irreducible algebraic curves, $d_X, d_Y$ are equal to the maximal cardinality of a fibre. Thus $|\mathcal{C}(A)| \leq d_X |A|$ and
\begin{align*}
    \frac{1}{d_Y} |A| \leq |\mathcal{C}(A)|
\end{align*}
if $A$ lies in the image of $\pi_Y$. In particular, the lower bound holds for all finite sets $A$ when $\pi_Y$ is surjective.

Let $L_X$ and $L_Y$ be line bundles on projective varieties $X$ and $Y$, respectively, and let $L = \pi_X^\ast L_X \otimes \pi_Y^\ast L_Y$. If $\mathcal{C}$ is a correspondence between $X$ and $Y$, then
\begin{align*}
    \deg_L(\mathcal{C}) =
    \deg_{L_X}((\pi_X)_\ast [\mathcal{C}]) +
    \deg_{L_Y}((\pi_Y)_\ast [\mathcal{C}]) =
    d_X \deg_{L_X}(X) + d_Y \deg_{L_Y}(Y)
\end{align*}
by the projection formula. Thus if $L_X$ and $L_Y$ are ample, then
\begin{align*}
    |\mathcal{C}(A)| \asymp_{\deg_L(\mathcal{C})} |A|
\end{align*}
for all finite sets $A$. In particular this will be true in our setup, described above in Section \ref{degrees}.

We will be working with correspondences between algebraic groups. 
If there is a correspondence between algebraic groups $G$ and $H$ that is the translate of an algebraic group, then $G$ and $H$ are isogenous. Let us now give two intuitive examples. 

\begin{example}
    We can fix a rational map $\varphi: X \rightarrow Y$ that is well-defined on an open $U \subseteq X$ and then consider $\mathcal{C}$ to be the Zariski-closure of the graph of $\varphi$. Bremner's question, discussed in the introduction, concerns the case where $X$ is an elliptic curve in Weierstrass form $y^2 = x^3 + ax + b$, $U = X \setminus \{O\}$ is $X$ without its point at infinity, $Y = \G_a$, and $\varphi(x, y) = x$. If $A \subseteq U(\C)$ is a finite set, then $\mathcal{C}(A)$ is the set of all $x$-coordinates occuring among points of $A$.
\end{example} 

\begin{example}
    Let $\varphi$ be a polynomial of degree $d \geq 1$. We can consider a correspondence between $\G_m \times \G_m$ whose complex points are given by $\{(x, \varphi(x)) : x \in \G_m(\mathbb{C}) \}$. One can see that the degree of this correspondence is $d + 1$. Moreover, $\mathcal{C}$ is a translate of an algebraic subgroup if and only if $\varphi$ is of the form $c x^d$ for some $c \in \C$. This is precisely the correspondence that we use for our deduction of Corollary \ref{shiftedbc} from Theorem \ref{thmdelta}.
\end{example}

\subsection{Mordell--Lang and $S$-unit equations. } \label{mordelllangsunit}

We recall here the deep results of Laurent, David--Philippon and Evertse--Schlickewei--Schmidt on the Mordell--Lang conjecture.  

\begin{thm}\label{ML}\cite{DavidPhilippon,Laurent,ESSlinear}  For any positive integers $d, g$, there exists a constant $C = C(d,g) \in \mathbb{N}$ with the following property. Suppose $G$ is an elliptic curve or $\mathbb{G}_m$. Let $\mathcal{V}\subseteq G^g$ be an algebraic variety of degree $d$ and $\mathcal{V}^{co}$ be 
$$\mathcal{V}^{co} = \mathcal{V}\setminus  \bigcup_{R + B \subset \mathcal{V}, \dim (B) > 0} (R + B),$$
where $R$ runs through points in $G^g$ and $B$ through connected algebraic subgroups. Then 
$$|\mathcal{V}^{co} \cap \Gamma|\leq C^{ 1 + r} $$
for any subgroup $\Gamma \subseteq G^g(\C)$ of rank $r$. More generally, one has
\begin{align*}
    \mathcal{V}\cap \Gamma = \bigcup_{i = 1}^{C^{1+r}} (\gamma_i+ H_i)\cap \Gamma,
\end{align*}
 where $\gamma_1, \dots, \gamma_{C^{1+ r}}$ are elements of $\Gamma$, and $H_1, \dots, H_{C^{1+r}}$ are connected subgroups of $G^g$ whose degrees are bounded in terms of $d$.
\end{thm}

\begin{proof} If $G$ is an elliptic curve then this theorem follows directly from \cite[Théorème 1.13]{DavidPhilippon}. 

For $G = \mathbb{G}_m$, we first prove the first part. We fix polynomials $Q_1, \dots, Q_k \in \C[X_1, \dots, X_g]$, such that $\mathcal{V}$ is their common zero-set. Their degree is bounded by the degree of $\mathcal{V}$ and the number of non-zero monomials in $Q_i$ is bounded by $\deg(Q_i)^g$ for $i = 1, \dots, g$. If we have a point $\gamma \in \mathcal{V}\cap \Gamma$, then $Q_i(\gamma) = 0, i = 1, \dots, k$ and if $\gamma \in \mathcal{V}^{co}$, then there is at least one $i$, such that no subsum of the monomials in $Q_i$ vanishes if evaluated at $\gamma$. This follows from the proof of Laurent \cite{Laurent}. The number of solutions of $Q_i(\gamma)= 0$ with no vanishing subsum is bounded by $c(\deg(Q_i),g)^{1 +r}$ for all $i$ \cite{ESSlinear}. This gives the first claim.

For the general statement we follow the proof of Laurent \cite{Laurent}. Each maximal algebraic group contained in $\mathcal{V}$ corresponds to a partition of the support of its defining equations. Thus, their number and degree is bounded only in terms of the degree of $\mathcal{V}$. For each algebraic subgroup given by a partition, Laurent constructs a map that reduces counting the number of intersection points to the $S$-unit equation for which we can apply the main theorem in \cite{ESSlinear}. 

Finally, the fact that the degree of each $H_i$ is bounded in terms of $d$ follows from the argument in \cite[Lemma 2]{Bombieri--Zannier}. 
\end{proof}

Note that $\mathcal{V}^{co}$ might be empty, even if $\mathcal{V}$ is not a coset. An easy example is a product $\mathcal{C} \times \mathbb{G}_m \subseteq \mathbb{G}_m^3$ for a curve $\mathcal{C}$, that is covered by cosets of the form $\{P\}\times \mathbb{G}_m$.

\section{ Projecting cartesian products} \label{pcp}

The main goal of this section is to prove Theorem \ref{projections} which describes expansion properties for certain projections of varieties. 

Thus, let $G$ to be some $1$-dimensional, connected algebraic group over $\mathbb{C}$, not isomorphic to $\mathbb{G}_a$, and let $\mathcal{B}$ be a projective variety of positive dimension. Let $\pi_1 : G^g \times \mathcal{B} \to G^g$ and $\pi_2: G^g \times \mathcal{B} \to \mathcal{B}$ be the canonical projection maps.  We recall the notion of a degenerate variety as described in Definition \ref{degenvar}.

\begin{definition}
We call an irreducible subvariety $\mathcal{V} \subseteq  G^g \times \mathcal{B}$ \emph{degenerate}, if there exists a connected algebraic group $H \subseteq G^g$ of positive dimension, and a proper subvariety $\mathcal{W} \subseteq G^g/H\times \mathcal{B} $ such that 
$$\mathcal{V} = \pi_H^{-1}(\mathcal{W}), $$
for the projection $\pi_H: G^g \times \mathcal{B} \rightarrow G^g/H\times \mathcal{B}$. 
    If $\mathcal{V}$ is not degenerate, we call $\mathcal{V}$ \textit{non-degenerate}. 
\end{definition} 

With this in hand, we now state our version of Theorem \ref{elekesronyai} for sets lying in low rank subgroups.

\begin{thm}\label{projections}
Let $\mathcal{V} \subseteq G^g\times \mathcal{B}$ be a non-degenerate subvariety of dimension $g$, such that $\pi_1$ and $\pi_2$ restricted to $\mathcal{V}$ are dominant. Let $\Gamma \subseteq G(\C)$ be subgroup of rank $r$. Then for any finite set $A \subseteq \Gamma$ we have
$$|\pi_2( (A^g\times \mathcal{B}) \cap \mathcal{V})| \geq c(g,\deg(\mathcal{V}))^{1 + r}|A|^g,$$
for a constant $c = c(g,\deg(\mathcal{V})) >0 $ depending only on $g$ and $\deg(\mathcal{V})$. 
\end{thm}

In order to prove Theorem \ref{projections}, we will require the following lemma.

\begin{lem} 
\label{cosets}
Suppose that $\mathcal{V} \subseteq G^g \times \mathcal{B}$ is a non-degenerate subvariety of dimension $g$, and suppose that the maps $\pi_1$ and $\pi_2$ are dominant. Then there exists a proper Zariski closed set $Z \subseteq G^g$, such that if there is a positive dimensional subgroup $H \subseteq G^g$ and $(P,Q) \in (G^g \times \mathcal{B})(\C)$ with 
\begin{align}\label{cosets} 
    \{(P + T, Q) : T \in H(\C) \} \subseteq \mathcal{V},
\end{align}
then $P + H \subseteq Z$. Moreover, the degree of the components of $Z$ and their number is bounded by a constant depending only on $g, \deg(\mathcal{V})$.


\end{lem}
\begin{proof} 
We first fix a connected algebraic subgroup $H$ of dimension $k$, and show that all $P$ such that $(P + H) \times \{Q\} \subseteq \mathcal{V}$ for some $Q$ are contained in a Zariski closed set $Z_H$, that depends on $H$. The lemma will be proved by taking a union of such sets $Z_H$. Let $p_H$ be the restriction of the quotient map $\pi_H : G^g \times \mathcal{B} \to (G^g/H) \times \mathcal{B}$ to $\mathcal{V}$. By Chevalley's theorem \cite[Theorem 1.3.1]{EGAIV} the set 
$$Z^H = \{y \in \mathcal{V}(\C) :  \dim(p_H^{-1}(p_H(y))) \geq k \}$$
is closed. Thus if $(P + H) \times \{Q\} \subseteq \mathcal{V}$ then $(P + H) \times \{Q\} \subseteq p_H^{-1}(p_H(P, Q))$, and so $(P, Q) \in Z^H$. 

Since $\mathcal{B}$ is projective, the projection $\pi_1 : G^g \times \mathcal{B} \to G^g$ is closed \cite[Theorem 1.11]{Shafarevic}. Therefore $Z_H = \pi_1(Z^H)$ is closed in $G^g$. If $Z_H = G^g$ then $Z^H$ has dimension $g$, and is therefore equal to $\mathcal{V}$. Also, the degree of $Z_H$ is bounded by the degree of $Z^H$, by the projection formula. It therefore suffices to show that $Z^H$ is not equal to $\mathcal{V}$ and that the degree of $Z^H$ is bounded in terms of $g$ and $\deg(\mathcal{V})$.

We will first prove that $Z^H \neq \mathcal{V}$, and so, suppose that $Z^H = \mathcal{V}$. Then consider $\mathcal{W}$, the Zariski closure of $\pi_H(\mathcal{V})$ and $Z' = \pi_H^{-1}(\mathcal{W})$, which is a subvariety of $G^g \times \mathcal{B}$ containing $\mathcal{V}$. Firstly, $\mathcal{W}$ is irreducible, because $\mathcal{V}$ is irreducible. Since $H$ is connected, the fibres of $\pi_H$ are irreducible, and so \cite[Theorem 1.26]{Shafarevic} implies that $Z'$ is irreducible.  
Since $\pi_H(\mathcal{V})$ is constructible it contains $U$ that is  Zariski--open (dense) in $\mathcal{W}$. We thus have that $Z' = \pi_H^{-1}(U) \cup \mathcal{E} $, where $\mathcal{E}$ is the a finite union of irreducible subvarieties $E = \pi_{H}^{-1}(E')$, with $E'$ running over all irreducible components of $\mathcal{W}\setminus U$. By the fibre dimension theorem $\dim(E) < \dim(Z')$ for all $E$. Since $\pi_H^{-1}(U)\subset \mathcal{V}$, we have that $Z' \subseteq \mathcal{V}\cup \mathcal{E}$, and a dimension count shows that $\dim(\mathcal{V}) = \dim(Z')$. Since both $\mathcal{V}$ and $Z'$ are irreducible $\mathcal{V} = Z'$.  
This means that 
$$\mathcal{V} = \pi_H^{-1}(\mathcal{W}) $$
and so $\mathcal{V}$ is degenerate. This contradicts our assumption on $\mathcal{V}$.

We will now prove that the degree of $Z^H$ is bounded in terms of $g$ and $\deg(\mathcal{V})$. For $g \geq 1$,  define $\exp_{G^g}$ the exponential of $G^g$ at the identity and $\mathcal{F}$ a suitably chosen fundamental domain for $\exp_G$. The graph $\exp_G$ restricted to $\mathcal{F}$ is a sub-Pfaffian set of complexity bounded by an absolute (effectively computable) constant, see work of Jones and the third author \cite{pfaffian}. Each algebraic group $H$, corresponds to a vector space $T_H$, such that $\exp_{G^g}(T_H) = H$. We then consider the set 
$$T^H = \{\gamma \in \mathcal{F}^g : \text{there exists } b \in \mathcal{B}(\C) \ \text{such that} \ \exp_{G^g}(\gamma  + T_H) \subseteq \mathcal{V} \cap ( G^g \times \{b\}) \},$$
which is a sub-Pfaffian set of complexity $c_{\text{comp}}$, where $c_{\text{comp}}$ depends only on $\deg(\mathcal{V})$.  We then have 
$Z^H = \exp_{G^{g}}(T^H)$, which has also bounded complexity, and it is a closed algebraic variety. The complexity of $Z^H$ bounds its degree, and so, we have proven this claim as well.

Finally, suppose $P + H$ is a maximal translate lying in the fibre $\mathcal{V}_Q = \pi_2^{-1}(Q)\cap \mathcal{V}$. Note that $\deg(\mathcal{V}_Q) \ll_g \deg(\mathcal{V})$ by B{\' e}zout's theorem. By an argument of Bombieri--Zannier \cite[Lemma 2]{Bombieri--Zannier}, if $H$ is an algebraic subgroup appearing in a maximal translate of $\mathcal{V}_Q$, then $H$ belongs to a finite set $\{H_1, \dots, H_\ell\}$ with $\ell \ll_{g, \deg(\mathcal{V})} 1$. The lemma is proved upon taking $Z = Z_{H_1} \cup \dots \cup Z_{H_\ell}$.
\end{proof}
In order to prove Theorem \ref{projections}, we combine Lemma \ref{cosets} with the estimates coming from uniform Mordell--Lang (Theorem \ref{ML}). In order to get control on the contribution of the closed set $Z$ from Lemma \ref{cosets}  we need the following Schwartz--Zippel type estimate. 
\begin{lem}\label{lemmaexception} Let $Z \subseteq G^g$ be an algebraic sub-variety. Then for any finite set $A \subseteq G(\C)$
$$|Z \cap A^{g}|\ll_{g, \deg(Z)}|A|^{\dim(Z)}.  $$
\end{lem}
\begin{proof}
We prove this by induction on the dimension. We may suppose that $Z$ is irreducible, since we can argue component wise. We can also pass to the closure $\overline{Z}$ of $Z$ in $\overline{G}^g$. If $\dim(Z) = 0$, this is trivial. So assume that $\dim(Z) \geq 1$. We can choose a factor $\overline{G}$ in $\overline{G}^g$ such that the projection from $\overline{Z}$ to $\overline{G}$ is surjective. Without loss of generality, we can assume this is the first factor. Then the intersection $\overline{Z} \cap (\{a\}\times \overline{G}^{g-1})$ has dimension equal to $\dim(Z) -1$. By B{\' e}zout's theorem 
\[ \deg(\overline{Z} \cap (\{a\}\times \overline{G}^{g-1}) )
\leq \deg(Z) \deg(\overline{G}^{g - 1})
\ll_{g}\deg(Z).\]
We then conclude by induction that
$$|Z\cap A^g|\leq \sum_{a \in A}|A^g \cap (\{a\}\times \overline{G}^{g-1})\cap \overline{Z}|\ll_{g, \deg(Z)}|A||A|^{\dim(Z) -1}. \qedhere$$
\end{proof}

We are now ready to prove Theorem \ref{projections}.

\begin{proof}[Proof of Theorem \ref{projections}] As $\mathcal{V}$ is non-degenerate, a coset contained in a fibre $\mathcal{V}_Q = \pi_2^{-1}(Q)\cap \mathcal{V}$ is contained in a closed set $Z \subseteq \mathcal{V}$ not depending on $Q$. This is Lemma \ref{cosets}. We set $A' = A^g \setminus Z(\C)$ and by Lemma \ref{lemmaexception}, $|Z \cap A^g| \ll_{\deg(\mathcal{V})}|A|^{g-1} $. Since the projection $\pi_1$ from $\mathcal{V}$ to $G^g$ is dominant and $\mathcal{B}$ is projective, $\pi_1$ is actually surjective, because it is closed \cite{Shafarevic}. Hence for each point $a \in A'$, there is a point $b \in B$ such that $(a,b) \in (A'\times \mathcal{B})\cap \mathcal{V}$. On the other hand, it follows from Theorem \ref{ML} that for every $b \in \mathcal{B}(\C)$, one has
$$|(A'\times\{b\})\cap \mathcal{V}| \leq c(r,g, \deg(\mathcal{V}))^{r + 1}$$
Thus the image $\pi_2((A^g\times \mathcal{B}) \cap \mathcal{V})$ contains at least 
$$c(r, g, \deg(\mathcal{V}))^{-1-r}(|A|^g - c'|A|^{g-1})$$
elements for a constant $c' > 0$ depending only on $\deg(\mathcal{V})$ and $g$, which finishes the proof of Theorem \ref{projections}.
\end{proof}

\section{Correspondences and cosets}\label{sectioncorr}

In this section we construct a variety $\mathcal{V}_{\text{sum}}$ with the property that $\pi_2((A^g \times H) \cap \mathcal{V}_{\text{sum}})$ is roughly the sumset $\mathcal{C}_1(A) + \dots + \mathcal{C}_g(A)$, for correspondences $\mathcal{C}_1, \dots, \mathcal{C}_g$ between algebraic groups $G$ and $H$. We would like to apply Theorem \ref{projections} to $\mathcal{V}_{\text{sum}}$, and so this section is dedicated to showing that this variety is non-degenerate, in the sense of Definition \ref{degenvar}.

Now let $G$ and $H$ be connected algebraic groups of dimension 1, such that $G$ is not isomorphic to $\mathbb{G}_a$. 
We compactify $H$ as described at the beginning of to section \ref{degrees}. Thus  $\overline{H}$ is either $\mathbb{P}^1$ or an elliptic curve, depending on whether $H$ is isomorphic to $\G_a$, $\G_m$ or to an elliptic curve. We also consider the $g$-fold sum map on $H$
\begin{align*}
    p_{\text{sum}}: H^g &\rightarrow H \\ (Q_1, \dots, Q_g) &\mapsto Q_1 + \cdots + Q_g,
\end{align*}
its graph $\Gamma(p_{\text{sum}}) \subseteq H^g \times H$, and its closure $\overline{\Gamma(p_{\text{sum}})}$ in $\overline{H}^{g + 1}$. 

Let $\mathcal{C}_1, \dots, \mathcal{C}_g \subseteq G \times \overline{H}$ be correspondences, none of which is the translate of an algebraic subgroup. We set 
$$ \mathcal{V}^* = \{(P_1, \dots, P_g, Q_1, \dots, Q_{g+1}): (Q_1, \dots, Q_{g+1}) \in \overline{\Gamma(p_{\text{sum}})}, (P_i, Q_i) \in \mathcal{C}_i\}, $$
which is an irreducible variety. The projection $\pi_{G^g \times \overline{H}}: G^g\times\overline{H}^{g + 1} \rightarrow G^g \times \overline{H}$ onto $G^g$ and the last coordinate of $\overline{H}^{g + 1}$ is a closed map by \cite[Theorem 1.11]{Shafarevic}. We set 
\begin{equation} \label{Vsumdef}
    \mathcal{V}_{\text{sum}} = \pi_{G^g\times\overline{H}}(\mathcal{V}^*)
\end{equation}
which is an irreducible variety since it is the image of an irreducible variety under a closed map. Thus the role of the projective variety $\mathcal{B}$ in Section \ref{pcp} is played by $\overline{H}$.

Our main goal in this section is to prove the following. 

\begin{proposition}\label{nondegenerate} The variety $\mathcal{V}_{\rm sum} \subseteq G^g\times\overline{H}$ is non-degenerate of dimension $g$. The projection $\pi_1 : \mathcal{V}_{\rm sum} \to G^g$ is surjective and the projection $\pi_2 : \mathcal{V}_{\rm sum} \to \overline{H}$ is dominant.  
\end{proposition}

As we exclusively work over complex algebraic groups, we will prove a lemma about holomorphic maps between tangent spaces of algebraic groups. 
\begin{lem}\label{lemfunctions} 
    Let $ U = U_1\times \cdots \times U_g$ be an open set of $\mathbb{C}^g$ and $f_i: U_i \rightarrow \mathbb{C}$ non-constant, holomorphic functions for $1 \leq i \leq g$. Suppose that there is a vector space $W \subseteq \mathbb{C}^g$ of dimension $1$ such that for any $b \in \C^g$, $(f_1, \dots, f_g)$ restricted to $(W + b) \cap U$ satisfies 
$$f_1(z_1) + \cdots + f_{g}(z_g) \equiv const.$$
Then there is at least one $j \in\{1, \dots, g\}$ such that $f_j$ is affine linear. 
\end{lem}
\begin{proof}
    After possibly permuting coordinates we can parameterise any co-set 
    $W + b$ by 
    $$(b_1, \dots, b_{g-k}, z,a_1 z + c_1,  \dots, a_{k-1}z + c_{k-1})$$ 
    where $k \geq 1$ is an integer, $a_1, \dots, a_{k-1} \in \C^*$ depend on $W$ and $b_1, \dots, b_{g-k}, c_1, \dots, c_{k-1}$ depend on $W + b$.  We may assume that $k \neq 1 $ since otherwise $f_g$ is constant. 
    We apply the invertible linear transformation 
\begin{align*}
L: \C^{g-k}\times \C\times \C^{k -1} & \rightarrow \C^{g-k}\times \C\times \C^{k -1} \\ 
    L(\underline{z},z, \underline{w}) & = (\underline{z},z, w_1 - a_1z, \cdots, w_{k-1}- a_{k-1}z)
\end{align*} 
to $U$ and the open set $L(U)$ then contains a product set $\tilde{U} = \tilde{U}_1\times \cdots \times \tilde{U}_g$. Taking the total derivative with respect to $z$ we obtain 
$$ \partial_{z_{g-k + 1}}f_{g-k + 1}(z) + a_1 \partial_{z_{g-k + 2}}f_{g-k + 2}(a_1z + c_1) + \dots + a_{k-1}\partial_{z_{g}}f_g(a_{k-1}z + c_{k-1}) =0,$$ 
for all $(b_1, \dots, b_{g-k},z, c_1, \dots,  c_{k-1}) \in \tilde{U}. $
Since $k \geq 2$, we may fix any 
\[(z_0, c_{1,0}, \dots, c_{k-2,0}) \in \tilde{U}_{g-k + 1}\times \cdots \times \tilde{U}_{g -1} \]
to find that $f_g$  is affine linear.  
\end{proof}

\begin{cor}\label{vs} Let $U_1\times \cdots \times U_g$ be an open set of $\mathbb{C}^g$ and $f_i: U_i \rightarrow \mathbb{C}$ holomorphic non-constant functions $i =1, \dots, g$. Suppose that there is a vector space $V\subsetneq \mathbb{C}^g$ of dimension $k $  such that for any $b \in \C^g$, $(f_1, \dots, f_g)$ restricted to $(V+ b) \cap U$ satisfies 
$$f_1(z_1) + \cdots + f_{g}(z_g) \equiv const.$$
Then  $f_j$ is affine linear for at least one $j \in \{ 1, \dots, g\}$. 
\end{cor}

\begin{proof}
    
    We can cover $V$ by translates of a one dimensional vector space $L$ and thus any translate of $V$ contains a translate of $L$. Thus Corollary \ref{vs} is implied by Lemma \ref{lemfunctions}.  
\end{proof}
\begin{proof}[Proof of Proposition \ref{nondegenerate}] 
Let $\mathcal{V}^o$ be the variety given by the points $(P_1, \dots, P_g, Q) \in G^g\times H$ such that   there exists  $(P_i, Q_i) \in \mathcal{C}_i\cap (G\times H(\C))$, with  $Q_1 + \cdots +Q_g = Q$. Note that $\mathcal{V}^o$
is open in $\mathcal{V}_{\text{sum}}$. For all but finitely many $P \in G(\C)$, there exists $Q \in H(\C)$ such that $(P, Q) \in \mathcal{C}_i(\C)$ for $i = 1, \dots, g$. Thus, $\pi_1$ is dominant and by \cite[Theorem 1.11]{Shafarevic} it is surjective. Also, for $Q \in H(\C)$, we can find $(Q_1, \dots, Q_g) \in H(\C)$ such that $Q_1 + \cdots + Q_g = Q$. It follows that $\pi_2$ is dominant. Now suppose that $\Vsum$ is degenerate, that is, there exists a connected algebraic group $H' \subseteq G^g$ of positive dimension, and a proper subvariety $\mathcal{W} \subseteq (G^g/H') \times \overline{H}$ such that 
\begin{equation} \label{bandeh}
    \mathcal{V}_{\rm sum} = \pi_{H'}^{-1}(\mathcal{W}), 
\end{equation}
for the projection $\pi_{H'}: G^g \times \overline{H} \rightarrow (G^g/H')\times \overline{H}$.

Now, let $s_1, \dots, s_g$ be  analytic functions on an open $U \subseteq G(\C))$ with target $H(\C)$, such that the graph of $s_i$ coincides with $\mathcal{C}_i(\C)$ restricted to $U\times H(\C)$. By (\ref{bandeh}) the sum $\sum_{i = 1}^gs_i$ (where we sum in $H$) is constant along $H' + P$ for all $P \in G^g(\mathbb{C})$. After perhaps shrinking $U$ to ensure that it is simply connected, we lift these functions to functions from the tangent space of $G$ to the tangent space of $H$, via setting $f_i = \exp_{H}\circ s_i\circ \log_G, i =1,\dots, g$. Now setting $V$  to be the tangent space of $H'$, and recalling that the sum $\sum_{i = 1}^gs_i$ is constant along translates of $H'$, we deduce that $f_1, \dots, f_g$ satisfy the conditions of Corollary \ref{vs}. Thus at least one of $f_i$ is affine linear. Lemma \ref{lemtangent} implies that at least one $\mathcal{C}_i$ is the translate of an algebraic subgroup, which contradicts  our assumption on the correspondences and concludes the proof. 
\end{proof}

\section{Freiman--type structural theorems} \label{addcom}

For the purposes of this section and the next, given a $1$-dimensional, connected algebraic group $H$ and some finite, non-empty set $A \subsetneq H$, we denote $\rk(A)$ to be the smallest integer $r \geq 1$ such that there exist $\xi_1, \dots, \xi_r \in H$ satisfying
\[ A \subseteq \{ n_1 \xi_1 + \dots + n_r \xi_r : n_1, \dots, n_r \in \mathbb{Z} \}. \]
The main aim of this section is to prove the following structural result.

\begin{lem} \label{wpfrlemma2}
        Let $H$ be a connected algebraic group of dimension $1$, let $A \subseteq H$ be a finite, non-empty set, let $n \geq 2$ be an integer such that $|nA| \leq K|A|$ for some $K > 1$. Then there exists some integer $1 \leq d$ and  some subset $A' \subseteq A$ such that 
        \[  d \ll 1+\frac{\log (4K)}{\log n} \ \ \text{and}  \ \ |A'| \gg \frac{|A|}{K^{\frac{C\log 2}{ \log n}}} \ \ \text{and} \ \  \rk(A') \leq d, \]
        where $C>0$ is some absolute constant.
\end{lem}

We will begin by proving the $n=2$ version of this. 

\begin{lem} \label{wpfrlemma}
        Let $H$ be a connected algebraic group of dimension $1$, let $A \subseteq H$ be a finite, non-empty set such that $|A+A| \leq K|A|$ for some $K > 1$. Then there exists some integer $1 \leq d \leq C\log (400K)$,  and some subset $A' \subseteq A$ such that $|A'| \geq |A|/(100K)^{C'}$ and $\rk(A') \leq d$,  where $C = 140$ and $C' = 110$.
\end{lem}

In order to prove Lemma \ref{wpfrlemma}, we will need the following very nice result of Gowers--Green--Manners--Tao \cite[Theorem 1.3]{GGMT2025} on the resolution of the weak polynomial Freiman--Ruzsa conjecture over $\mathbb{Z}$.

\begin{lem} \label{wpfr}
Let $D$ be a positive integer, let $A \subsetneq \mathbb{Z}^D$ be a finite, non-empty set such that $|A+A| \leq K|A|$ for some $K >1$. Then there exists some integer $1 \leq d \leq C \log (4K)$, some elements $\vec{x}_1, \dots, \vec{x}_d \in \mathbb{Z}^D$ and some subset $A' \subseteq A$ such that $|A'| \geq |A|/K^{C'}$ and 
\[ A' \subseteq \{ n_1 \vec{x}_1 + \dots + n_d \vec{x}_d : n_1, \dots, n_d \in \mathbb{Z}\}, \]
where $C = 140$ and $C' = 110$.
\end{lem}

We will also need the following simple lemma. 

\begin{lem} \label{mst}
Let $H$ be a connected algebraic group of dimension $1$, let $A \subsetneq H$ be a finite, non-empty set. Then the subgroup generated by $S$ is isomorphic to some subgroup of $\mathbb{Z}^D \times \mathbb{Z}/n \mathbb{Z} \times \mathbb{Z}/m\mathbb{Z}$, for some non-negative integer $D$ and some $n,m\in \mathbb{N}$.
\end{lem}

\begin{proof}
    This is true when $H = (\mathbb{C}, +)$ since any finitely generated subgroup of $(\mathbb{C}, +)$ is isomorphic to $\mathbb{Z}^D$ for some $D \in \mathbb{N}$. This is slightly more non-trivial when $H = (\mathbb{C}^*, \cdot)$, but it is a standard fact that any finitely generated subgroup of $(\mathbb{C}^*, \cdot)$ is isomorphic to $\mathbb{Z}^D \times \mathbb{Z}/N \mathbb{Z}$ for some non-negative integer $D$ and some $N \in \mathbb{N}$. Finally, when $H$ is some elliptic curve over $\mathbb{C}$, we may use the fact, mentioned in Section 3.1, that $H$ is isomorphic to $\mathbb{C}/L$, where $L$ is some lattice of rank two in $\mathbb{C}$, to deduce that any finitely generated subgroup of $H$ is isomorphic to some subgroup of $\mathbb{Z}^D\times \mathbb{Z}/n\mathbb{Z} \times \mathbb{Z}/m\mathbb{Z}$ for some non-negative integer $D$ and some $n,m\in \mathbb{N}$.
\end{proof}

We are now ready to prove Lemma \ref{wpfrlemma}

\begin{proof}[Proof of Lemma \ref{wpfrlemma}]
    Let $\Gamma$ be the subgroup generated by $S$. We can use Lemma \ref{mst} to view $\Gamma$ as a subgroup of $\mathbb{Z}^l \times \mathbb{Z}/n\mathbb{Z} \times \mathbb{Z}/ m \mathbb{Z}$ for some integers $n,m \geq 1$ and some integer $l \geq 0.$  Now, for any $0\leq i, j \leq 9$, define
    \[ A_i = \{ x \in \mathbb{Z} : in/10 \leq x < (i+1)n/10 \} \ \ ({\rm mod} \ n)  \] 
    and
    \[ B_j = \{ x \in \mathbb{Z} : jm /10 \leq x < (j+1)m/10 \} \ \ ({\rm mod} \ m) . \]
    Thus $\mathbb{Z}/n \mathbb{Z} \times \mathbb{Z}/m\mathbb{Z} = \cup_{0 \leq i,j \leq 9} (A_{i} \times B_j)$. Moreover, let $S_{i,j} = S \cap (\mathbb{Z}^l \times A_i \times B_j)$ for every $0 \leq i,j \leq 9$. Since
    \[ \sum_{0 \leq i,j \leq 9} |S_{i,j}|  = |S|,\]
    by the pigeonhole principle, there exist some $0 \leq i,j \leq 9$ such that $|S_{i,j}| \geq |S|/100$.

    Let $\pi : \mathbb{Z}^l \times \mathbb{Z}/n \mathbb{Z} \times \mathbb{Z}/m \mathbb{Z} \to \mathbb{Z}^{l+2}$ be the map satisfying
    \[ \pi(\vec{x}, a \ ({\rm mod} \ n),b\ ({\rm mod} \ m)) = (\vec{x},a,b) \]
    for all $\vec{x} \in \mathbb{Z}^l$ and $a \in \{0,1,\dots, n-1\}$ and $b \in \{0,1,\dots, m-1\}$. We now claim that for any $s_1, s_2,s_3,s_4 \in S_{i,j}$, one has
    \begin{equation} \label{freiman2iso}
        \pi(s_1) + \pi(s_2) = \pi(s_3) + \pi(s_4) \ \ \text{if and only if} \ \  s_1 + s_2 = s_3 + s_4. 
    \end{equation}
    In order to see this, first note that since $\pi^{-1}$ is just the projection map, it suffices to check that equality on the right hand side implies equality on the left hand side. Writing 
    \[ s_l = (\vec{x}_l, a_l \ ({\rm mod} \ n),b_l\ ({\rm mod} \ m)) \]
    for every $1 \leq l \leq 4$, we see that $s_1 + s_2 = s_3 + s_4$ implies that 
    \[ a_1 + a_2 - a_3 - a_4  \equiv 0\ ({\rm mod} \ n) \ \ \text{and} \ \ b_1 + b_2 - b_3 - b_4 \equiv 0 \ ({\rm mod} \ m). \]
    Since $in/10 \leq a_1, a_2, a_3, a_4 < (i+1)n/10$, we see that 
    \[ a_1 + a_2 - a_3 - a_4 \in [-n/5, n/5]\cap \mathbb{Z}.\] The preceding congruence condition now necessitates that $a_1 + a_2 = a_3 + a_4$. A similar argument gives us that $b_1 + b_2 = b_3 + b_4$.

    Thus, writing $S_1 = \pi(S_{i,j})$, the equivalence in \eqref{freiman2iso} implies that
    \[ |S_1 + S_1| = |S_{i,j} + S_{i,j}| \leq |S +S| \leq K|S| \leq 100K |S_{i,j}| = 100K |S_1| \]
    Since $S_1 \subseteq \mathbb{Z}^{l+2}$, we may now apply Lemma \ref{wpfr} to find some subset $S_1' \subseteq S_1$ such that $|S_1'| \geq |S_1|/(100K)^{C'}$ and 
    \[ S_1' \subset \{ n_1 \vec{x}_1 + \dots + n_d \vec{x}_d : n_1, \dots, n_d \in \mathbb{Z}\},\]
    where $\vec{x}_1, \dots, \vec{x}_d \in \mathbb{Z}^{l+2}$ are some elements and $1\leq d \leq C \log (400K)$ is some integer. This implies that
    \[ \pi^{-1}(S_1') \subset \{ n_1 \pi^{-1}(\vec{x}_1) + \dots + n_d \pi_1^{-1}(\vec{x}_d) : n_1, \dots, n_d \in \mathbb{Z} \}. \]
    Setting $S' = \pi^{-1}(S_1')$ finishes the proof of Lemma \ref{wpfrlemma}.
\end{proof}

We now present our proof of Lemma \ref{wpfrlemma2}.

\begin{proof}[Proof of Lemma \ref{wpfrlemma2}]
 Since $n \geq 2$, we have that $|2S| \leq |nS|$, and so, whenever $2 \leq n \leq 16$, we may apply Lemma \ref{wpfrlemma2} and adjust the implicit constant in the Vinogradov notation to obtain the desired result. Thus, we assume that $n > 16$, in which case, writing $k = \lfloor (\log n)/(\log 2) \rfloor$, we see that $k \geq 4$. Now since $|2^k S| \leq |nS | \leq K |S|$, we get that
 \[ K \geq \frac{|2^k S|}{|S|} = \prod_{1 \leq j \leq k} \frac{|2^j S|}{|2^{j-1}S|}, \]
 whence, there exists some $1 \leq j \leq k$ such that
 \begin{equation} \label{boot1}
|2^j S| =  |2^{j-1}S + 2^{j-1} S| \leq K^{1/k} |2^{j-1} S|. 
 \end{equation}
 Applying Lemma \ref{wpfrlemma} for the set $2^{j-1}S$, we get that there exists some set $X \subseteq 2^{j-1}S$ such that
 \begin{equation} \label{boot2}
 |X| \gg \frac{|2^{j-1} S|}{ K^{C/k} } \ \ \text{and} \ \ X \subseteq \{ n_1 \xi_1 + \dots + n_d \xi_d : n_1, \dots, n_d \in \mathbb{Z} \} ,
 \end{equation}
 for some 
 \[ d \ll 1 + \frac{ \log (4K) }{k} \]
and some points $\xi_1, \dots, \xi_d \in H$.

 Thus, it suffices to prove that there exists a large subset of $S$ which is contained in a translate of the set $-X$. In order to do this, note that
 \[ |X||S| = \sum_{y \in X+ S}  |(y-X) \cap S| \leq |X+S| \max_{y \in X+S} |(y-X)\cap S|.  \]
 Hence, it suffices to show that 
 \[ \frac{|X + S|}{|X|} \ll K^{C'/k},\]
 for some absolute constant $C'>0$. In order to see this, we combine the fact that $X \subseteq 2^{j-1} S$ along with \eqref{boot1} and \eqref{boot2} to get that
\[ 
      \frac{|X+S|}{|X|} \leq \frac{|2^{j-1}S + 2^{j-1}S|}{|X|} \ll K^{C/k} \frac{|2^j S|}{|2^{j-1}S|} \ll K^{(C+1)/k}.
\]
 This concludes our proof of Lemma \ref{wpfrlemma2}.
\end{proof}

We briefly remark that structural results akin to Freiman's inverse theorem are often used in unison with covering results. One such result is known as Ruzsa's covering lemma. 

\begin{lemma} \label{rcl}
    Let $G$ be an abelian group, let $A, B \subseteq G$ be non-empty sets such that $|A+B| \leq K|B|$. Then there exists some non-empty set $X \subseteq A$ such that $|X| \leq K$ and $A \subseteq X + B - B.$  
\end{lemma}

This immediately combines with Lemma \ref{wpfrlemma} to deliver the following result. 

\begin{lemma} \label{coveringlemma}
     Let $H$ be a connected algebraic group of dimension $1$, let $A \subseteq H$ be a finite, non-empty set such that $|A+A| \leq K|A|$ for some $K > 1$. Then there exists some integer $1 \leq d \leq C\log (400K)$, some finite subset $T \subsetneq H$ with $\rk(T) = d$ and some non-empty $X \subseteq H$ such that 
     \[ |X| \leq (100K)^{C'+1} \ \ \text{and} \ \  A \subseteq X + T,\]
     where $C = 140$ and $C' = 110$.
\end{lemma}

\section{Proofs of main results} \label{alltheproofs}

In this section, we present the proofs of all of our results mentioned in \S\ref{introduction} and \S\ref{furtherapp}. 

We begin by deducing Theorem \ref{thmBremner} from Corollary \ref{corBremner}. 

\begin{proof}[Proof  of Theorem \ref{thmBremner}]
Setting $G $ to be an elliptic curve in Weierstrass form (\ref{Weierstrass}), $H = \G_a$ and $\mathcal{C}$ the correspondence given by $(x,y,x)$, we can apply  Corollary \ref{corBremner} to obtain the first part of Theorem \ref{thmBremner}. We can proceed similarly with $H = \G_m$ to obtain the desired conclusion for geometric progressions. Setting $\mathcal{C}$ equal to $(x,y,z), x = (u + z)^2$ we get the bound on successive squares. Also, for example choosing $z^2 = x$, we can also bound the length of arithmetic progressions in certain higher genus curves. 
\end{proof}

We will now prove Theorem \ref{thmdelta} by combining Theorem \ref{rankversion} and Lemma \ref{wpfrlemma2}.

\begin{proof}[Proof of Theorem \ref{thmdelta}]
    Let $k \geq 1$ and $d \geq 2$ be integers, let $g \in \mathbb{N}$ be sufficiently large in terms of $d,k$.  We may further assume that $|gA| \leq |A|^k$ since otherwise we would be done. In this case, we may apply Lemma \ref{wpfrlemma2} to find some $\xi_1, \dots, \xi_r \in G$ and some $A' \subseteq A$ such that $r \ll k \log |A|/ \log g$ and
    \[|A'| \gg |A|^{1 - Ck/\log g}  \ \ \text{and} \ \  A' \subseteq \{ n_1 \xi_1 + \dots + n_d \xi_d : n_1, \dots, n_d \in \mathbb{Z} \}. \]
    The latter condition implies that the subgroup generated by $A'$ has rank at most $r$, and so, we may apply Theorem \ref{rankversion} to deduce that
    \begin{align*}
    |\mathcal{C}_1(A) + \dots + \mathcal{C}_g(A)| &  \geq |\mathcal{C}_1(A') + \dots + \mathcal{C}_{2k}(A')| \\
    & \geq c(d,2k)^{-1 - r} |A'|^{2k} \\
    & \gg_{k} c(d,2k)^{-1} |A|^{-\frac{k \log c(d,2k)}{\log g}} |A'|^{2k} \\
    & \gg_{k,d} |A|^{2k - \frac{2Ck^2}{\log g} - \frac{k\log c(2k,d)}{\log g}}.
    \end{align*}
    Choosing $g$ to be sufficiently large so as to ensure that 
    \[ \frac{2Ck^2}{\log g} < k/2 \ \ \text{and}  \ \ \frac{k\log c(d,2k)} {\log g} < k/2 \ \ \text{and} \ \ 2k < g,\]  
    we get that
    \[ |\mathcal{C}_1(A) + \dots + \mathcal{C}_g(A)| \geq |\mathcal{C}_1(A) + \dots + \mathcal{C}_{2k}(A)| \gg_{k,d} |A|^{k}. \]
    This finishes the proof of Theorem \ref{thmdelta}.
\end{proof}

Theorem \ref{esz} is a special case of Theorem \ref{elsz}. Indeed, if $\mathcal{V}$ is an irreducible subvariety of $G^g$ which is not a coset of a subgroup, then we have the trivial inequality ${\rm codef}(\mathcal{V}) \leq {\rm dim}(\mathcal{V})-1$. Thus, we will now prove Theorem \ref{elekesronyai}.

    

\begin{proof}[Proof of Theorem \ref{elekesronyai}]
Since $|A+A| \leq K|A|$, we may apply Lemma \ref{coveringlemma} to deduce that $A \subseteq Y +T$, where $T$ is contained in a subgroup $\Gamma$ of $H$ with rank $d \ll \log K$ and $Y \subset H$ satisfies $|Y| \leq K^{O(1)}$. Now, since $X \subseteq A$, this means that $X$ is also contained in at most $K^{O(1)}$ translates of some subgroup $\Gamma$. By the pigeonhole principle, we can find some $X'\subseteq X$ with $|X'| \geq |X|/K^{O(1)}$ such that $X'$ is contained in a translate of $\Gamma$, and so, $X'$ is contained in a subgroup of rank $d+1$. We now apply Theorem \ref{projections} to deduce that
\begin{align*}
    |\pi_2( (X^g \times \mathcal{B}) \cap \mathcal{V})| & \geq |\pi_2( (X'^g \times \mathcal{B}) \cap \mathcal{V})| \geq \frac{|X'|^g}{C^{2 + d}} \\
    & \geq \frac{|X|^g}{K^{O(g)} 2^{C' \log K}} \gg \frac{|X|^g}{K^{C''}}
\end{align*}
where $C,C', C''>0$ are constants depending only on $g, \deg(\mathcal{V})$.
\end{proof}

Next, we prove Theorem \ref{rankversion}.

\begin{proof}[Proof of Theorem \ref{rankversion}]
We consider the variety $\mathcal{V}_{\rm sum}$ as described in \eqref{Vsumdef} and note that for any finite set $A \subset G$, the set 
\[ \pi_2( (A^g \times \overline{H} ) \cap \mathcal{V}_{\rm sum}) = \mathcal{C}_1(A) + \dots + \mathcal{C}_g(A). \]
Proposition \ref{nondegenerate} implies that this variety $\mathcal{V}_{\rm sum}$ is non-degenerate of dimension $g$ and the projections $\pi_1$ to $G^g$ is surjective and $\pi_2$ to $\overline{H}$ restricted to $\mathcal{V}$ is dominant. Thus, we may apply Lemma \ref{projections} to deduce that 
\[ |\mathcal{C}_1(A) + \dots + \mathcal{C}_g(A)| \geq c(g, \deg(V))^{1+r}|A|^g. \qedhere \]
\end{proof}

We will now deduce Corollary \ref{corBremner} from Theorem \ref{rankversion}.

\begin{proof}[Proof of Corollary \ref{corBremner}]
Throughout this proof, let $C_1, C_2, C_3>0$ be positive constants depending only on $d$. Suppose that $A$ is a generalised arithmetic progression in $\mathcal{C}(\Gamma)$ of rank $k$ for some $k \in \mathbb{N}$. We will first show that $k \leq C^{1 + r}$ for some constant $0 < C \ll_d 1$. In order to see this, let 
\begin{equation} \label{hypercube}
P' = \{P_0 + \ell_1P_1 + \cdots +\ell_kP_\ell  :  0\leq \ell_i \leq 1\} . \end{equation}
Since $L_1, \dots, L_k \geq 2$, we get that $P' \subseteq P$. Since $P$ is proper, we get that $P'$ is also proper, and so, one has
\[ |P'+P'| \leq 3^k = |P'|^{\frac{\log 3}{\log 2}}. \]
Applying Theorem \ref{rankversion} for $P'$, we get that
\[ |P'|^2 C_1^{-1 -r} \ll |P'+P'| \ll |P'|^{\frac{\log 3}{\log 2}},\]
whence $2^k = |P'| \leq C_2^{1 + r}$. Now, we consider the set $P$ and use the preceding upper bound on $|P'|$ to observe that
\[ |P+P| \leq (2L_1-1) \dots (2L_k - 1) \leq 2^k L_1 \dots L_k = 2^k |P| \leq C_2^{1 + r}|P|. \]
 We can now apply Theorem \ref{rankversion} to deduce that 
    $$C_1^{-1 - r}|P|^2 \leq  |P+P| \leq C_2^{1 + r}|P|, $$
and so, we obtain the desired claim $|P| \leq C_3^{1 + r}$.
\end{proof}

We now consider Theorem \ref{sumprod1}. As in the proof of Theorem \ref{thmdelta}, we will derive this by putting together Theorem \ref{rankversion} and Lemma \ref{wpfrlemma2}.

\begin{proof}[Proof of Theorem \ref{sumprod1}]
    Let $A \subseteq G$ be a finite, non-empty set. We may assume that $|A+A| \leq |A|^{1 + \delta}$, since otherwise we would be done. In this case, we apply Lemma \ref{wpfrlemma} to obtain $\xi_1, \dots, \xi_r \in G$ such that 
    \[ r \leq 140( \delta \log|A| + \log 400) \]
    and some subset $A' \subseteq A$ satisfying
    \[ |A'| \geq  \frac{|A|}{(100 |A|^{\delta})^{110}} \ \ \text{and} \ \ A' \subseteq \{ n_1 \xi_1 + \dots + n_r \xi_r : n_1, \dots, n_r \in \mathbb{Z} \}. \]
    The latter condition implies that $A'$ is contained in a subgroup of rank $\leq r$ whence we may apply Theorem \ref{rankversion} to deduce that
    \begin{align*}
        |\mathcal{C}(A') + \mathcal{C}(A')| &  \geq c(d,2)^{ -(r+1)} |A'|^2 \\
        & \geq c |A|^{2-\delta(140  \log c(d,2) + 220)}.
     \end{align*}
    Combining this with the fact that $|\mathcal{C}(A) + \mathcal{C}(A)| \geq |\mathcal{C}(A') + \mathcal{C}(A')|$ then delivers the claimed estimate.
\end{proof}

We prove Theorem \ref{elsz} by combining Theorem \ref{ML} and Lemma \ref{coveringlemma}.

\begin{proof}[Proof of Theorem \ref{elsz}] Since $\mathcal{V}$ is not a co-set, we have for every subgroup $\Gamma$ of rank $r$, the equality 
$$\mathcal{V}\cap \Gamma = \bigcup_{i = 1}^S (\gamma_i + H_i)\cap \Gamma, $$
where the $H_i$ are connected subgroups whose degree is bounded in terms of the degree of $\mathcal{V}$, and $S \leq c^{1 + r}$, where $c = c(\deg(\mathcal{V}), g)>0$ is a constant, see Theorem \ref{ML}. 

As $|A + A|\leq K|A|$, there exists $1 \leq d \leq C\log(400K)$, some finite subset $T \subsetneq G$, with $\text{rk}(T) = d$ and some non-empty subset $X \subsetneq G$, such that $|X| \leq (100K)^{C' + 1}$,  and $A \subseteq X + T$, where $C= 140, C' = 110$, see Lemma \ref{coveringlemma}. We deduce that $A^g$ is contained in the translate of a finite subset $T_g$ of  rank $gd$ translated by a set $X_g$ of cardinality $(100K)^{gC' + g}$. Let $\Gamma_A$ be the group generated by the elements in $T_g$. Note that the rank $r$ of $\Gamma_A$ satisfies 
\[ r \leq gd \ll g \log (100K). \]
Let $x \in X_g$. From Theorem \ref{ML} it follows that $\gamma \in (\mathcal{V}-x )\cap\Gamma_A$ either lies in a co-set of degree bounded in terms of the degree of $\mathcal{V}$ or in  a set of cardinality $c^{1 + r}$. The number of co-sets for fixed $x$ is bounded by $c^{1 + gr}$ and so applying Lemma \ref{lemmaexception} to each co-set contained in $\mathcal{V}$, we get the desired bound.
\end{proof}

\appendix

\section{Diophantine equations } \label{appa}
We show that our results related to Bremner's conjecture have some consequences for the Mordell--Lang conjecture. Roughly speaking, the theorems of David--Philippon and Evertse--Schmidt--Schlickewei give very uniform bounds on the number of points on a variety that lie on a finitely generated group. However, they generally do not provide effective bounds for the height of these points. The situation is, in a strong sense even more dire for elliptic curves. The points on an elliptic curve with coordinates in a number field form a finitely generated group, but there is no known algorithm to determine its generators. It is straightforward to pass from a number estimate to a (partly) effective version of Mordell--Lang. For example for a given curve $C$ with a bound $t$ on the number of rational points $C(K)$ (with $K$ being some number field), we can easily say that the rational points on $C^{t+1}(K)$ lie on a finite union of proper subvarieties given by setting some coordinates equal to each other. However, this is still a far cry from actually determining the Zariski-closure of $C^{t + 1}(K)$. 
Our work does not resolve this issue in general but, fixing a number field $K$, and allowing for a very high dimensional power of an elliptic curve $E^t$, we can construct large families of surfaces, for which we can determine the Zariski-closure of their rational points.   
In what follows we let $E$ be an elliptic curve in  Weierstrass form given by
$$ y^2 = x^3 + ax+ b,$$
with $a,b \in K$, where $K$ is a number field. Let $r$ be the rank of $E(K)$.  
 A tuple $(a_1, a_2, \dots, a_t)$ forms an arithmetic progression if and only if it is a point on the plane $P$ in $\mathbb{A}^t$ defined by the equations
\begin{align}
    Z_{j + 2} - 2Z_{j + 1} + 
    Z_j = 0. \label{arith}
\end{align}
for $j \in \{1, \dots, t - 2\}$.

Now consider the subvariety of $(\mathbb{A}^2)^t \times \mathbb{A}^t$ defined by the equations
\begin{align*}
    Y_j^2 - X_j^3 - aX_j - b = 0
\end{align*}
for $j \in \{1, \dots, t\}$ and
\begin{align*}
    Z_{j + 2} - 2Z_{j + 1} 
    + Z_j = 0
\end{align*}
for $j \in \{1, \dots, t - 2\}$. Such a variety is merely the product $U^t \times P$ of $t$ copies of an affine elliptic curve
\begin{align*}
    U : y^2 = x^3 + ax + b
\end{align*}
with the plane $P$, and thus has dimension $t + 2$. The points on this variety correspond to $t$-tuples of points of an elliptic curve and $t$-term arithmetic progressions, with no relation between them.

Finally we impose algebraic relations
\begin{align*}
    P_i(X_i, Y_i, Z_i) = 0, ~~ i = 1, \dots, t
\end{align*}
between the points on the elliptic curve, and the terms of the arithmetic progression. So we also ask that $P_i \notin K[X_i, Y_i]\cup K[Z_i]$ is irreducible such that they induce a correspondence. For example, one can take
\begin{align*}
    P_i(X_i, Y_i, Z_i) =
    X_i - Z_i^2,~~ i = 1, \dots, t,
\end{align*}
which encodes the condition that the $x$-coordinates of the points on $U$ should be squares of elements of an arithmetic progression. Each of these new relations $P_i$ decreases the dimension by one, because they only involve the variables $(X_i, Y_i, Z_i)$. It follows that the subvariety of $U^t \times P$ defined by the $t$ new relations is a surface $S_A$.
A similar construction works with the quadratic equations 
\begin{align}
    Z_{j + 1}Z_1= Z_jZ_2, ~~ j = 2, \dots, t-1.\label{geom} 
\end{align}
We get a variety $U^t \times Q$ and imposing the algebraic conditions given by $P_i$ result in a surface $S_G$.
We denote by $D \subsetneq S_A, S_B$, the subvariety, that is given by the additional equation $Z_1 = Z_2$. Thus the points of $D$ correspond to degenerate arithmetic or geometric progressions.

This is all slightly technical but for $P_i = Z_i - X_i$ the rational point $S_A(K)$ correspond to arithmetic progressions in $E(K)$ of length $t$ and similarly for $S_G(K)$ and geometric progressions. It is not hard to see, and we give the details in the paragraph below, that these surfaces $S_A, S_B$ are then of general type. The Mordell--Lang conjecture predicts that the Zariski-closure of $S_A(K), S_B(K)$ consists of a finite union of elliptic curves and points. This  implies that there are only finitely many arithmetic or geometric sequences of length 3 in an elliptic curve. 

In this setting, Corollary \ref{corBremner} delivers a more precise version that exactly predicts the distribution of rational points, albeit for $t$ depending on $K$. 

\begin{thm}\label{BombieriLang}
   For each $d,r \geq 0$, there exists and effectively computable $t$, such that $S_A(K) =S_B(K)=  D(K)$.  
\end{thm}

The variety $D$ is either the empty set or a finite union of copies of the elliptic curve $E$. 
We can embed $\mathbb{A}^{3t}$ into $\mathbb{P}^{3t}$ via 
$$(X_1,Y_1 \dots ,X_t, Y_t, Z_1, \dots, Z_t) \hookrightarrow [X_1,Y_1,\dots, X_t, Y_t,Z_1, \dots, Z_t, 1] $$
and the Zariski-closure $\overline{S}_A$ of $S_A$ is a projective surface. Let $d_1, \dots, d_t$ be the degrees $P_1, \dots, P_t$ and if $\overline{S}_A$ is smooth, then the degree of the canonical class of $\overline{S}_A$ is 
$$d_1 + \cdots + d_t + t- 3 , $$
see \cite[Examples 5.1.1]{HindrySilverman}. Thus $\overline{S}_A$ is of general type. A similar argument works for $S_G$. 

\section{Degenerate polynomials} \label{Appb}

Let $G, \mathcal{B}, \pi, \pi_2$ be as in Definition \ref{degenvar}.   We denote $P \in \mathbb{C}[x_1, \dots, x_g]$ to be degenerate with respect to $\mathbb{G}^g$ if the variety $\mathcal{V} \subseteq \mathbb{G}_a^g \times \mathcal{B}$ defined by the equation $P(x_1, \dots, x_g) = t$ is degenerate.   In this section, we briefly comment on possible ways to check whether a given polynomial $P \in \mathbb{C}[x_1, \dots, x_g]$ is degenerate with respect to $G^g$ when $G = \G_a$ or $G = \G_m$. In the latter case, it was shown by the second author \cite{Mu2024a} that the polynomial
\[ \sum_{\vec{\alpha} \in E} c_{\vec{\alpha}} x_1^{\alpha_1} \dots x_g^{\alpha_g}, \]
where $E \subseteq \mathbb{C}^g$ is a finite, non-empty set and $c_{\vec{\alpha}} \neq 0$ for all $\vec{\alpha} \in E$, is non-degenerate if and only if $\{\sum_{\vec{\alpha} \in E} z_{\vec{\alpha}} \cdot \vec{\alpha} : z_{\vec{\alpha}} \in \mathbb{C} \} = \mathbb{C}^g$. 

We will now provide a criterion to check whether a polynomial $P \in \mathbb{C}[x_1, \dots, x_g]$ is non-degenerate with respect to $\mathbb{G}_a^g$.


\begin{lemma} \label{degenequiv}
    If $P \in \mathbb{C}[x_1, \dots, x_g]$ is degenerate with respect to $\mathbb{G}_a^g$, then there exists non-zero $\vec{v} \in \mathbb{C}^g$ such that the identity
    \[ \vec{v} \cdot (( \nabla P)(x_1, \dots, x_g)) = 0\]
    holds. 
\end{lemma}
\begin{proof}
Let $\mathcal{V}$ be given by the variety $P(\vec{x}) = t$, where, by abuse of notation, we denote $\vec{x} = (x_1, \dots, x_g)$. Since $G = \mathbb{G}_a$ and $\mathcal{V}$ is degenerate, one can deduce that 
\[ P(\vec{x}) = F(L_1, \dots, L_k)\]
for some $0 \leq k < g$ and some linear forms $L_1, \dots, L_k \in \mathbb{C}[x_1, \dots, x_g]$ and some $F \in \mathbb{C}[y_1, \dots, y_k]$. Since $k < g$, there exists some non-zero $\vec{v} = (v_1, \dots, v_g) \in \mathbb{C}^g$ such that the identity 
\[ L_i(\vec{x}) = L_i(\vec{x} + t \cdot \vec{v})\]
holds for all $1 \leq i \leq k$ and all $t \in \mathbb{R}$. Here $t \cdot \vec{v} = (tv_1, \dots, tv_g)$. In particular, we have the identity
\[ P(\vec{x}) = P(\vec{x} + t \cdot \vec{v}).\]
Differentiating with respect to $t$, we get that
\begin{align*}
0   = \sum_{j=1}^k   \left(  \frac{\partial F }{\partial y_j} \right) (L_1 (\vec{x} + t \cdot \vec{v}), \dots, L_k(\vec{x} + t \cdot \vec{v}))  \cdot \sum_{l=1}^g \left( \frac{\partial L_j}{\partial x_l} \right)(\vec{x} + t \cdot \vec{v}) \cdot v_l 
\end{align*}
for all $t \in \mathbb{R}$. Setting $t = 0$ and rearranging the sums, we get that
\begin{align*}
0 & = \sum_{l=1}^g v_l  \left( \sum_{j=1}^k \left( \frac{\partial F }{\partial y_j} \right) (L_1 (\vec{x}), \dots, L_k(\vec{x})) \cdot  \left( \frac{\partial L_j}{\partial x_l} \right)(\vec{x}) \right)  = \vec{v} \cdot ( (\nabla P)(\vec{x})  ). \qedhere
\end{align*}
\end{proof}

We will use this to prove that the polynomial $P(x,y,z) = xy + yz + zx$ is non-degenerate with respect to $\G_a^3$. Indeed, if $P$ were to be degenerate with respect to $\mathbb{G}_a^3$, then Lemma \ref{degenequiv} implies that there would exist some non-zero $\vec{v} = (v_1, v_2, v_3) \in \mathbb{C}^3$ such that the identity
\[ \vec{v}\cdot ( (\nabla P)(x,y,z)  ) = v_1 (y+z) + v_2 (x+y) + v_3(z+x) = 0 \]
holds. In particular, this would mean that
\[ v_1 + v_2 = v_2 + v_3 = v_3 + v_1 = 0,\]
that is, $v_1 = v_2 = v_3 = 0$. This contradicts the hypothesis that $\vec{v}$ is non-zero, and so, $P$ must be non-degenerate with respect to $\G_a^3$.

\bibliographystyle{amsplain} 
\bibliography{arxivbib}

\end{document}